\documentclass[12pt]{article}  
\usepackage{amsmath,amscd,amsthm,amsfonts,amssymb} 
\usepackage{pgfpages}
\usepackage{adjustbox}
\usepackage{graphicx}
\usepackage{tikz}
\newtheorem{defn}{Definition} 
\newtheorem{rmk}{Remark} 
\newtheorem{thm}{Theorem} 
\newtheorem{prop}{Proposition} 
\newtheorem{lem}{Lemma}

\newtheorem{cor}{Corollary} 
\newtheorem{ass}{Assumption} 
\newcounter{constant}

\newtheorem*{mainthm}{Main Theorem}

\usepackage[section]{placeins}
\FloatBarrier
\begin{document}
	\title{Lissajous-toric knots} 
	\author{Marc Soret and Marina Ville} 
	\date{ } 
	\maketitle

\begin{abstract} A point in the $(N,q)$-torus knot in $\mathbb{R}^3$ goes $q$ times along a vertical circle while this circle rotates $N$ times around the vertical axis. In the Lissajous-toric knot $K(N,q,p)$, the point goes along a vertical Lissajous curve (parametrized by $t\mapsto(\sin(qt+\phi),\cos(pt+\psi)))$ while this curve rotates $N$ times around the vertical axis. Such a knot has  a natural braid representation $B_{N,q,p}$ which we investigate here. If $gcd(q,p)=1$, $K(N,q,p)$ is ribbon; if $gcd(q,p)=d>1$, $B_{N,q,p}$ is the $d$-th power of a braid which closes in a ribbon knot. We give an upper bound for the $4$-genus of $K(N,q,p)$ in the spirit of the genus of torus knots; we also give examples of $K(N,q,p)$'s which are trivial knots.
	\end{abstract}

	\section{Introduction} 
	We study a class of knots generalizing torus knots, which we call {\it Lissajous-toric}: a torus knot is generated by a  a circle rotating around an axis and a Lissajous-toric knot is generated by a Lissajous curve rotating around the axis. There are several ways of describing them.
	\subsection{Lissajous-toric knots: various points of view} 
	\subsubsection{A description in $\mathbb{R}^3$}\label{a description in R3}	
	We recall the description of the  $(N,q)$-torus knot in  $\mathbb{R}^3$ endowed with an orthonormal frame $Oxyz$ (see for exemple [Cr] 1.5). If $\Gamma$ is the circle of radius $1$ centered at $(0,2,0)$ in the $yz$ plane, a point travelling along the knot  goes $q$ times around $\Gamma$ while $\Gamma$ is rotated $N$ times around the axis $Oz$. \\
	 In the case of the {\it Lissajous-toric knots}, we replace the vertical circle by a vertical Lissajous curve: we take three integers $N,q,p$ with $(N,q)=(N,p)=1$ and a real number $\phi$, and we define a knot $K(N,q,p,\phi)$ as follows. Consider the curve $C_{q,p,\phi}$ given in a vertical plane by  $$t\in[0,2\pi]\longrightarrow\mathbb{R}^3$$ $$:t\mapsto \Big(0,2+\sin\big(qt\big),\cos\big(p(t+\phi)\big)\Big)$$ and rotate $C_{q,p,\phi}$ is rotated $N$ times around the axis generated by $(0,0,1)$. In Cartesian coordinates, we write the knot as
	\[ (*)\left \{\begin{array}{lcr}
	x=\big(2+\sin(qt)\big)\cos(Nt) \\
	y=\big(2+\sin(qt)\big)\sin(Nt)\\
	z=\cos\big(p(t+\phi)\big)\end{array} \right. \]

	\subsubsection{A description in the $3$-dimensional cylinder}
	We write (*) above in cylindrical coordinates:
	\[ \left\{\begin{array}{lcr}
	\theta=Nt \\
	\rho=2+\sin(qt)\\
	z=\cos\big(p(t+\phi)\big)\end{array} \right. \] 
	Thus $K(N,q,p,\phi)$ is a closed $N$-braid which we can write in the $3$-cylinder $\mathbb{S}^1\times\mathbb{R}^2$ as follows : 
	\begin{equation}\label{definition du disque} 
	e^{it}\mapsto \Big(e^{Nit}, \sin\big(qt\big), \cos\big(p(t+\phi)\big)\Big) 
	\end{equation} 
	Note the similarity with the $(N,q)$-torus knot which is written in the $3$-sphere or the $3$-cylinder as
	\begin{equation}
	e^{it}\mapsto (\frac{1}{\sqrt{2}}e^{Nit}, \frac{1}{\sqrt{2}}e^{qit})
	\end{equation} 

		\subsubsection{Billiard curve in a solid torus}
	\begin{center} 
	 \begin{figure}[!h]
\includegraphics[scale = .4 ]{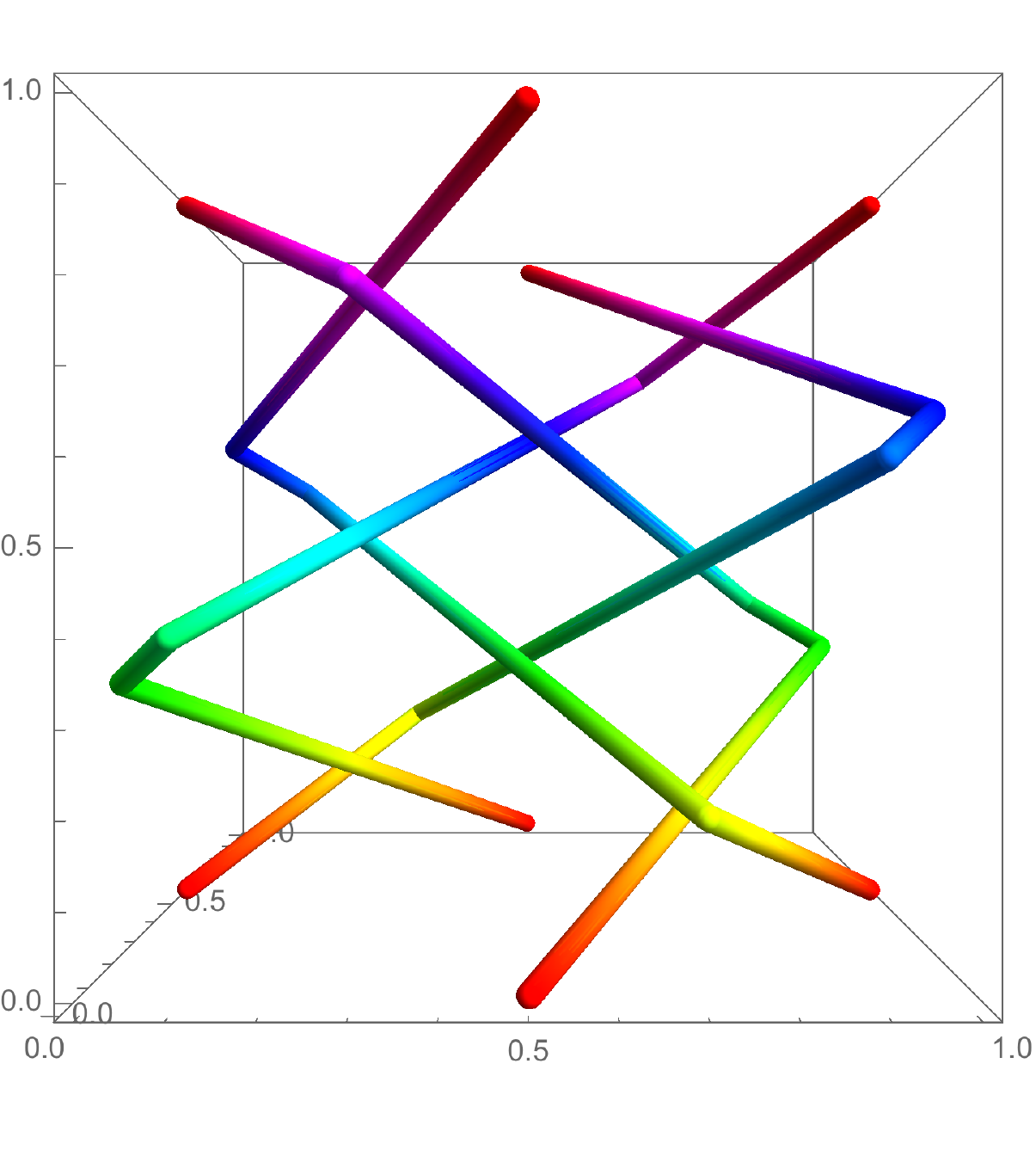}
\includegraphics[scale = .4]{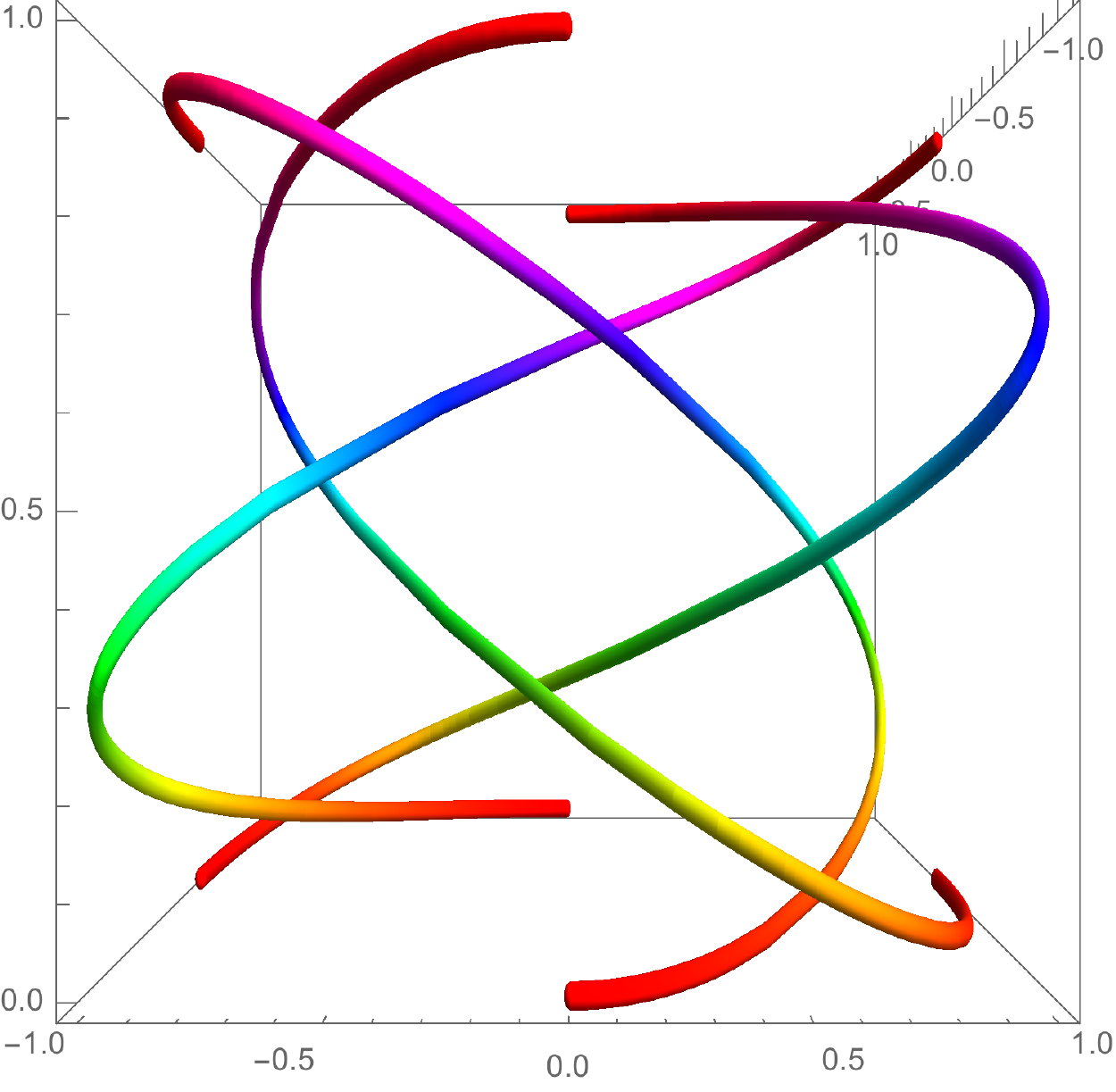}
 
\caption{\label{fig1} Perspective front  view of  knots $B(3,4,5)$ and  $K(3,4,5)$ }
\end{figure}
        \end{center}
	Just as Billiard curves are equivalent to Lissajous knots (cf. [JP])  Lissajous toric knots are equivalent to billiard curves in a   square solid torus, namely a cube where the top and bottom have been identified; C. Lamm introduced them in [La 1], see also the related [L-O].
	Such billiard curves are parametrized similarly to Lissajous toric knots; the trigonometric functions are replaced by  {\it saw-tooth functions}  of the type $g(t) := 2|t-[t]-\frac{1}{2}|.$  and 
	  $h(t) :=  t - [t] $.
	\begin{equation}
	  C(N,p,q, \phi)  : \left( 
	      \begin{array}{cc}
	         [0,2\pi]  &\longrightarrow [-1,1]^3 \\
	               t & \mapsto   \left( g(N.t), g(p.t+\phi), h(q.t) \right)
	       \end{array}  
	  \right)
	\end{equation}

	C. Lamm  noticed that these billiard curves in a  solid torus do not depend on the phase up to mirror transformation and stated that, if $p$ and $q$ are mutually prime, the knot $K(N,q,p)$ is ribbon.\\
 
	\subsubsection{Singularity knots of minimal surfaces}

	We first encountered the $K(N,q,p,\phi)$'s in [S-V] when  we studied the singularities of minimal disks in $\mathbb{R}^4$;  having noticed that their knot types do not depend on the phase $\phi$ up to mirror transformation,  we dropped the $\phi$ in the notation.\\
	
	We consider a minimal, i.e. conformal harmonic, map $F:\mathbb{D}\longrightarrow \mathbb{R}^4$ where $\mathbb{D}$ is the unit disk in $\mathbb{C}$, with $dF(0)=0$, i.e. $F$ has a branch point at $0$. If moreover $F$ is a topological embedding, we can copy Milnor's construction of algebraic knots ([Mi]) and take the intersection of $F(\mathbb{D})$ with a small sphere centered at $F(0)$: we obtain  a {\it minimal knot}. Complex curves are a special case of minimal surfaces and the germ $z\mapsto (z^N,z^q)$ yields the $(N,q)$-torus knot. 
	In  [S-V] the knots $K(N,q,p,\phi)$'s came from germs of singularities of the type 
	\begin{equation}\label{germ}
	z\mapsto \big(Re(z^N), Im(z^N), Im(z^q),Re(e^{pi\phi}z^p)\big)
	\end{equation}
	with
	\begin{equation}\label{hypothese riemannienne} 
	N < p,q
	\end{equation}
	In [S-V] we called the $K(N,q,p)$'s {\it simple minimal knots}; in the present paper we   drop the assumption (\ref{hypothese riemannienne}) and study these knots {\it per se}; {\it Lissajous-toric} is a more appropriate name for the general case.

	\subsection{Contents of the paper}\label{subsection: Contents of the paper} In [S-V] we defined a braid $B_{N,q,p}$ naturally associated to the knot $K(N,q,p)$; we describe it here in much greater detail. We view $B_{N,q,p}$ as a collection of graphs of $N$ functions from $[\eta,1+\eta]$ to $\mathbb{R}^2$; the purpose of the small positive number $\eta$ is to avoid crossing points at the endpoints of the interval.
	\begin{figure}[!h]
		\begin{center}
	\fbox{\includegraphics[width= 7 cm, height = 5 cm]{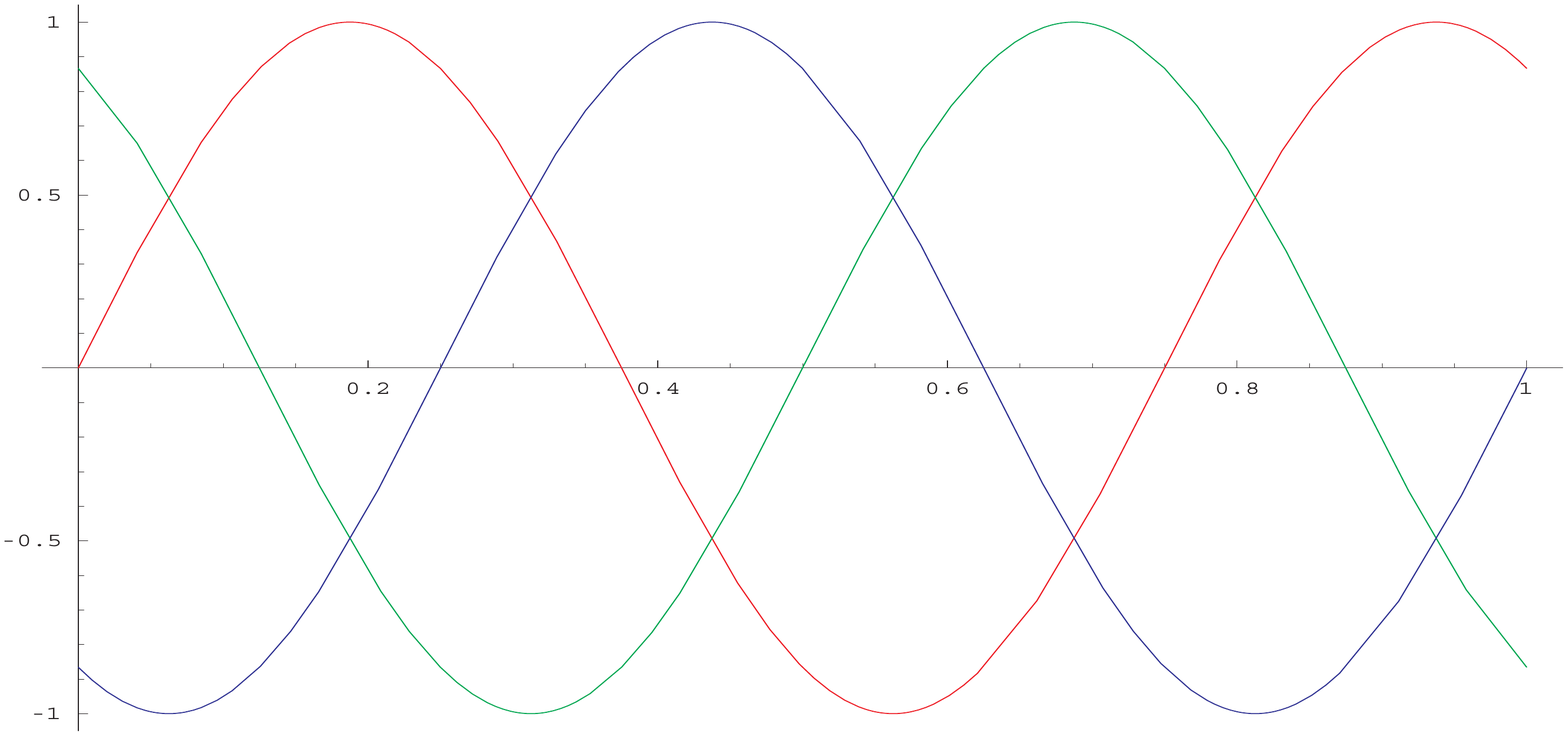}}
	\caption{Braid shadow of $B(5,q,p)$}\label{premiere image de tresse}
	\end{center}
	\end{figure}

	 We prove in \S \ref{preuve de la proposition sur les tresses periodiques} below
	 \begin{prop}\label{ce qui nous mene a l'assumption}
	  Let $d=gcd(p,q)$,  $\tilde{q}=\frac{q}{d}$, $\tilde{p}=\frac{p}{d}$ ;   then
	  \begin{equation}\label{B est une puissance de B tilda} 
	  B_{N,q,p}=B_{N,\tilde{q},\tilde{p}}^d \end{equation} \end{prop} 
	  
	  Since $\tilde{q}$ and $\tilde{p}$ are mutually prime and since the knot type does not change if we interchange $p$ and $q$, we make the	  \begin{ass}\label{assumption on the integers} 	The numbers $p$ and $q$ are mutually prime and $q$ is odd. 	\end{ass} In \S \ref{paragraph - description of the simple minimal braid}, we construct two braids $\alpha_{N,q,p}$ and $\beta_{N,q,p}$ of the form \begin{equation}\label{la premiere fois qu'on voit alpha et beta} \alpha_{N,q,p}=\prod_{2\leq 2k\leq N-1}\sigma_{2k}^{\pm}\ \ \ \ \beta_{N,q,p}=\prod_{1\leq 2k+1\leq N-1}\sigma_{2k+1}^{\pm} \end{equation} where the $\sigma_i$'s are the standard generators of the braid group ${\bf B}_N$ and the exponents $\pm 1$ of the $\sigma_i$'s appearing in $\alpha$ and $\beta$ are given by simple formulae in $N,q,p$.\\ \\
	  We will state below the Main Theorem which  expresses   the braid $B_{N,q,p}$ as a product of the braids $\alpha_{N,q,p}$, $\alpha_{N,q,p}^{-1}$, $\beta_{N,q,p}$ and $\beta_{N,q,p}^{-1}$ as follows : \begin{equation}\label{la premiere fois qu'on voit B comme un commutateur} B_{N,q,p}= Q_{N,q,p}\alpha_{N,q,p}Q_{N,q,p}^{-1}\beta_{N,q,p} \end{equation} where the $N$-braid $Q_{N,q,p}$ is also a product of $\alpha_{N,q,p}^{\pm 1}$'s and $\beta_{N,q,p}^{\pm 1}$'s. We illustrate the Main Theorem in \S \ref{section sur les exemples}   by going through the examples we gave in [S-V]  and we prove it in \S \ref{The structure of a simple minimal braid: proofs}. \\
	
	In the rest of the paper, we drop the Assumption \ref{assumption on the integers} and study the topology of the knot. In \S \ref{preliminaires sur le ribbon}, we prove a theorem stated by Lamm
	\begin{thm}\label{theoreme sur le ribbon}

		If $p$ and $q$ are mutually prime, the knot $K(N,q,p)$ is ribbon.
		\end{thm}
		\begin{cor}\label{corollaire sur les noeuds periodiques}
		If $d=gcd(p,q)>1$, the knot $K(N,q,p)$ is periodic and its braid is the $d$-th power of a braid which closes in a ribbon knot.
	\end{cor}

	\begin{thm}
		\label{theoreme sur le four-genus}
		If $d=gcd(p,q)$, the four-genus of $K(N,q,p)$ verifies
\begin{equation}\label{inegalite pour le four genus}
		g_4(K(N,q,p))\leq\frac{(N-1)(d-1)}{2}.
		\end{equation}
	\end{thm}
		\begin{rmk}
		 The right-hand side of (\ref{inegalite pour le four genus}) is the genus of the $K(N,d)$-torus knot (cf.  [K-M]).
		\end{rmk}
		
		\begin{rmk} 
		The inequality (\ref{inegalite pour le four genus}) can be strict: for example the knot $K(3,5,10)$ is $10_{123}$ which is slice.
		\end{rmk}
		There is one case where we know that  (\ref{inegalite pour le four genus}) is an equality:
		\begin{prop}\label{four genus; egalite} Let $N,q,p$ be positive integers with $(N,q)=(N,p)=1$, $d=gcg(q,p)$ and let 
			\begin{equation} \tilde{p}=\frac{p}{d}\ \ \ \ \ \tilde{q}=\frac{q}{d} 
			\end{equation} 
			\begin{equation}\label{congruence} \mbox{If}\ \ \ \ \ \  \tilde{p}+\tilde{q}\equiv 0\ (2N)\ \ \ \mbox{or}\ \ \ \ \ \  \tilde{p}-\tilde{q}\equiv  0\ (2N) 
			\end{equation} the knot $K(N,q,p)$ is represented by a quasipositive braid and its $4$-genus is \begin{equation}\label{egalite pour le genre} g_4\big(K(N,q,p)\big)=\frac{(N-1)(d-1)}{2} 
			\end{equation} 
			 \end{prop} 

	Finally, replacing  $t$ by $t+\pi$ in the expression of $K(N,q,p)$ given in \S \ref{a description in R3} yields
	\begin{prop}
		If $p$ and $q$ have different parities (and thus $N$ is odd), then $K(N,q,p)$ is preserved by the involution
		$$(x,y,z)\mapsto (-x,-y,-z).$$
		Hence it is positive strongly amphicheiral. 
	\end{prop}  
	Some of the $K(N,q,p)$'s are actually trivial knots; in \S \ref{la section sur les noeuds triviaux} show:
\begin{prop}\label{liste de noeuds triviaux}
	If $N$ and $q$ are mutually prime, the knots $K(N,q,q+N)$, $K(N,q,1)$, $K(N,q,2Nq+1)$ and $K(N,q,2Nq-1)$ are trivial.
\end{prop}
 Can we get all the trivial $K(N,q,p)$'s this way? We did computer simulations using the braid software from the Liverpool knot group ([br]) and KnotPlot ([KP]): they told us that in some cases (the $K(4,5,.)$'s for example) the answer is yes but in most cases the answer is no (see the lists of Jones polynomials at the end of the paper).

    \section*{Acknowledgment} 	We are grateful to Moshe Cohen whose stimulating conversation prompted us to embark on this work. \section{The structure of a simple minimal braid}\label{paragraph - description of the simple minimal braid} \subsection{Overview} 

   Here is an informal description of the contents of the Main Theorem.\\ \\ There are $2q$ values of $t$ in $[\eta, 1+\eta]$ (we call them {\it crossing values}) above which two or more of the $N$ graphs forming $B_{N,q,p}$ meet (at {\it crossing points}) and the data of all these crossing points make up the braid (see Figure \ref{premiere image de tresse}); above each of the crossing values $t$'s, the generators of the braid group ${\bf B}_N$ describing the corresponding crossing points are all even (i.e. of the form $\sigma_{2k}^{\pm 1}$) or all odd (i.e. of the form $\sigma_{2k+1}^{\pm 1}$).\\ \\ The set of crossing points of $B_{N,q,p}$ above a crossing value $t$ can be represented  by one of the braids: $\alpha_{N,q,p}$, $\alpha_{N,q,p}^{-1}$, $\beta_{N,q,p}$ or $\beta_{N,q,p}^{-1}$ which were introduced in formula (\ref{la premiere fois qu'on voit alpha et beta}); thus $B_{N,q,p}$ is a product of the $\alpha_{N,q,p}$'s and $\beta_{N,q,p}$'s and of their inverses.\\ We order the $2q$ crossing values $t_1<t_2<...<t_q<...<t_{2q}$. Going from $t_k$ to $t_{k+1}$ changes $\alpha_{N,q,p}^{\pm 1}$ into $\beta_{N,q,p}^{\pm 1}$ or vice-versa. A formula gives us the exponent $+1$ or $-1$ of the $\alpha_{N,q,p}$ or $\beta_{N,q,p}$ above a given crossing point $t_k$ in terms only of $N,q,p$ and $k$.\\ Finally we notice that, if we have an $\alpha_{N,q,p}$ (resp. $\alpha_{N,q,p}^{-1}$, $\beta_{N,q,p}$, $\beta_{N,q,p}^{-1}$) for $t_k$ (with $k\neq q$), we have a $\alpha_{N,q,p}^{-1}$ (resp. $\alpha_{N,q,p}$, $\beta_{N,q,p}^{-1}$, $\beta_{N,q,p}$) for $t_{2q-k}$: this explains the presence of $Q_{N,q,p}$ and $Q_{N,q,p}^{-1}$ in the product (\ref{la premiere fois qu'on voit B comme un commutateur}).
 \subsection{Statement of the structure theorem}
 \label{subsection: Statement of the structure theorem}

 	  \begin{mainthm}\label{structure generale} Let $N,p,q$ be three integers such that $q$ is odd and $(p,q)=(N,q)=(N,p)=1$; and let $A$, $B$ two integers such that 
 \begin{equation} \label{definition des A et B} 2NA+Bq=1 
 \end{equation} For  $i\in\{1,...,N-1\}$, we let \begin{equation}\label{definition d'epsilon avec le signe de Conway} \epsilon_{N,q,p}(i)=(-1)^{[\frac{pBi}{N}]} \end{equation} 
 where $[\ ]$ denotes the integral part and we define \begin{equation} \label{deuxieme apparition de alpha et beta} \alpha_{N,q,p}=\prod_{2\leq 2i\leq N-1}\sigma_{2i}^{\epsilon_{N,q,p}(2i)}\ \ \ \beta_{N,q,p}=\prod_{1\leq 2i+1\leq N-1}\sigma_{2i+1}^{\epsilon_{N,q,p}(2i+1)} \end{equation} For  $k\in\{1,...,2q\}$, $k\neq q, k\neq 2q$, we let \begin{equation}\label{definition de lambda avec le signe de Conway} \lambda_{N,q,p}(k)=(-1)^{[\frac{2Apk}{q}]} 
 \end{equation} Up to mirror transformation, the knot $K(N,q,p)$ is represented by the braid 
 \begin{equation} 	\label{expression comme un commutateur} 	B_{N,q,p}=\underbrace{\alpha_{N,q,p}^{\lambda(1)}\beta_{N,q,p}^{\lambda(2)}... \alpha_{N,q,p}^{\lambda(q-2)}\beta_{N,q,p}^{\lambda(q-1)}}_{Q_{N,q,p}}\alpha_{N,q,p}\underbrace{\beta_{N,q,p}^{-\lambda(q-1)}\alpha_{N,q,p}^{-\lambda(q-2)}...\beta_{N,q,p}^{-\lambda(2)}\alpha_{N,q,p}^{-\lambda(1)}}_{Q_{N,q,p}^{-1}}\beta_{N,q,p} 	\end{equation}
  The $k$-th factor in this expression corresponds to the $k$-th crossing value $t_k$. \end{mainthm} 
 Notice that the arithmetic formulae (\ref{definition d'epsilon avec le signe de Conway}) and (\ref{definition de lambda avec le signe de Conway}) can be written in terms of the Conway sign:
  \begin{defn} ([Co])	\label{definition du signe de Conway} If $m$ and $n$ are two  integers, $n$ is said to be {\it positive} (resp. {\it negative}) modulo $m$ if $n$ is congruent to an integer inside $(0,\frac{m}{2})$ (resp. $(0,-\frac{m}{2})$).
  \end{defn}

\section{Illustrations and examples}\label{section sur les exemples} 
In this section we go through the examples featured in [S-V] and we write their braid using the terminology  of the Main Theorem.\\
We define three permutations of the crossing values $t_k, \ k\in \{ 1,\ldots, 2q\}$ 
and their  corresponding action  on the  blocks  $\alpha$ and $\beta$ :

\begin{equation}
\begin{array}{l}
T(k) = k + q,  S(k) = 2q-k , R(k) = q-k  \\
T : \alpha \mapsto \beta,  \beta \mapsto \alpha\\
S : \alpha \mapsto \alpha^{-1},  \beta \mapsto \beta^{-1}\\
R : \alpha \mapsto \beta^{-1},  \beta \mapsto \alpha^{-1}
\end{array}
\end{equation}



 %

 Since it is clear in each case of the following list what the $N,q,p$ are, we dropped the indices $N,q,p$.
 \begin{itemize} 
 	\item $N=3, q=4, p=5$: square knot $3_1\# \bar{3}_1$ $$Q\sigma_2^{-1}Q^{-1}\sigma_1\ \ \ \ \ \mbox{where}\ \ \ \ \ Q=\sigma_2\sigma_1^{-1}\sigma_2^{-1}\sigma_1$$ 
 	\item $N=3, q=4, p=7$: trivial knot $$Q\sigma_2Q^{-1}\sigma_1^{-1}\ \ \ \ \ \mbox{where}\ \ \ \ \ Q=(\sigma_2^{-1}\sigma_1^{-1}\sigma_2^{-1})^2$$ 
 	\item $N=3, q=4, p=10$: figure eight knot $$B_{3,4,10}=B_{3,2,5}^2=
 	(Q\sigma_2Q^{-1}\sigma_1^{-1})^2 \ \ \ \ \mbox{where}\ \ \ \ \ Q=\sigma_2\sigma_1\sigma_2\sigma_1$$ 
 	\item $N=3, q=5, p=7$: $10_{155}$ 
 	$$Q\sigma_2^{-1}Q^{-1}\sigma_1^{-1}\ \ \ \ \ \mbox{where}\ \ \ \ \ Q=\sigma_2^{-1}\sigma_1\sigma_2^{-1}\sigma_1$$ 
 	Note that this knot verifies the assumptions of Theorem \ref{four genus; egalite}. 
 	\item $N=3, q=5, p=10$: $10_{123}$ $$B_{3,5,10}=B_{3,1,2}^5=(\sigma_2^{-1}\sigma_1)^5$$ 
 	\item $N=3, q=7, p=8$: $5_1\# \bar{5}_1$ $$Q\sigma_2^{-1}Q^{-1}\sigma_1\ \ \ \ \ \mbox{where}\ \ \ \ \ Q=\sigma_2\sigma_1^{-1}\sigma_2\sigma_1\sigma_2^{-1}\sigma_1$$ 
 	\item $N=3, q=7, p=19$: $14N11995$ $$Q\sigma_2Q^{-1}\sigma_1\ \ \ \ \ \mbox{where}\ \ \ \ \ Q=\sigma_2\sigma_1^{-1}\sigma_2^{-1}\sigma_1\sigma_2\sigma_1^{-1}$$ 
 	\item $N=4, q=5, p=7$: $5_2\# \bar{5}_2$
 	$$Q\alpha Q^{-1}\beta\ \ \ \ \ \mbox{where}\ \ \ \alpha=\sigma_2^{-1}\ \ \ \ \beta=\sigma_1\sigma_3\ \ \ \ Q=\alpha^{-1}\beta^{-1}\alpha\beta$$ 
 	\item $N=4, q=5, p=13$: $9_{46}$
 	$$Q\alpha Q^{-1}\beta\ \ \ \ \ \mbox{where}\ \ \ \alpha=\sigma_2\ \ \ \ \beta=\sigma_1\sigma_3\ \ \ \ Q=\alpha\beta\alpha^{-1}\beta^{-1}$$ 
 	\item $N=5, q=6, p=22$: $7_7$ $$B_{5,6,22}=B_{5,3,11}^2=(Q\alpha Q^{-1}\beta)^2\ \ \ \ \ \mbox{where}\ \ \ \alpha=\sigma_2\sigma_4^{-1}\ \ \ \ \beta=\sigma_1^{-1}\sigma_3\ \ \ \ Q=\alpha^{-1}\beta$$ 
 \end{itemize}

 \begin{figure}[!h]
 	\begin{center} 
 			\includegraphics[angle=-90, scale = .4 ]{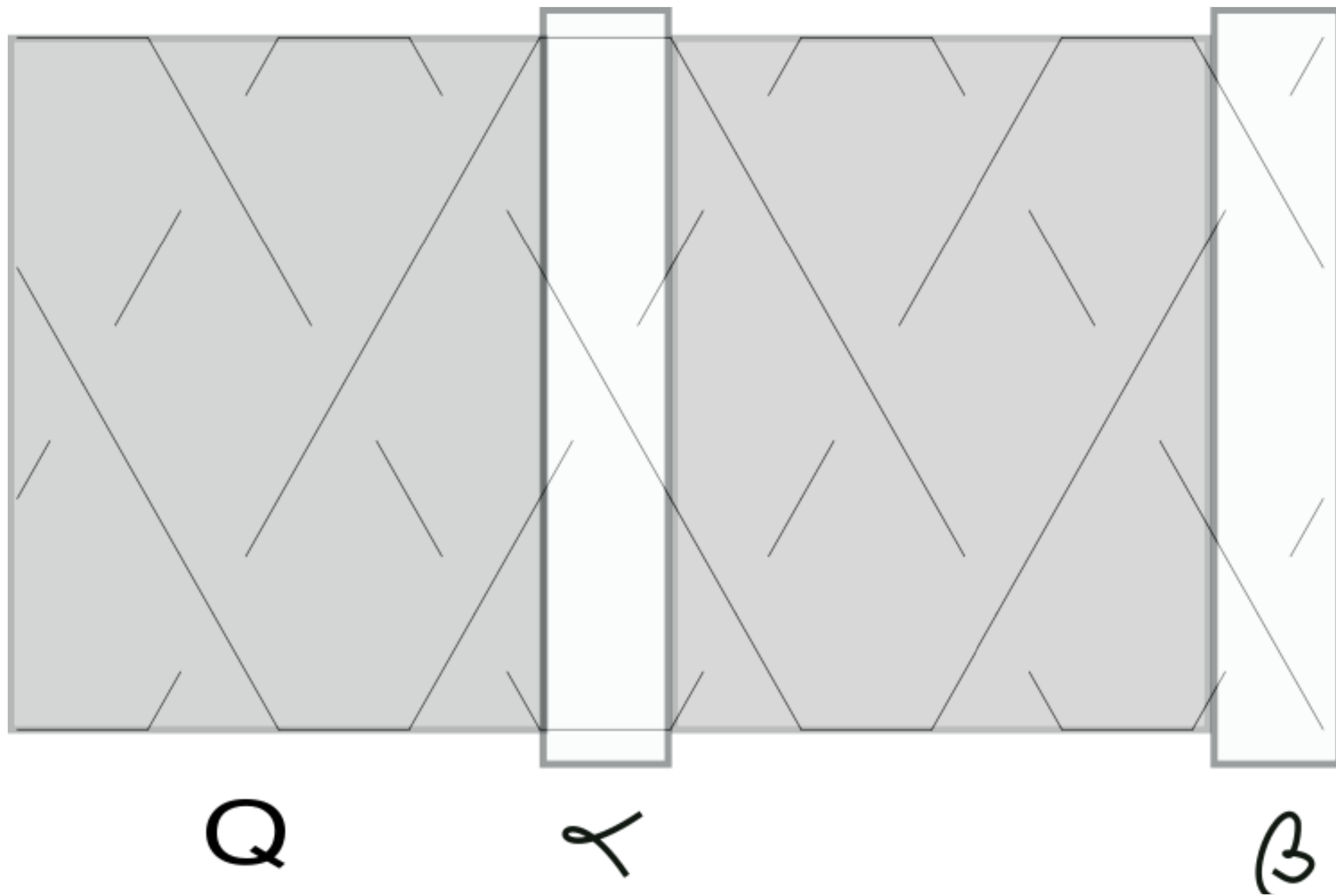}

 		\caption{   $B_{4,5,13}$ , $A=4$, $\alpha=\sigma_2, \beta=\sigma_1\sigma_3$, $\lambda (k) = (-1)^{[ \frac{2k}{5} ]}$} 
 			
 	\end{center}
 \end{figure}

 \begin{figure}[h!]
 	\begin{center}
 		
 \includegraphics[scale = .4]{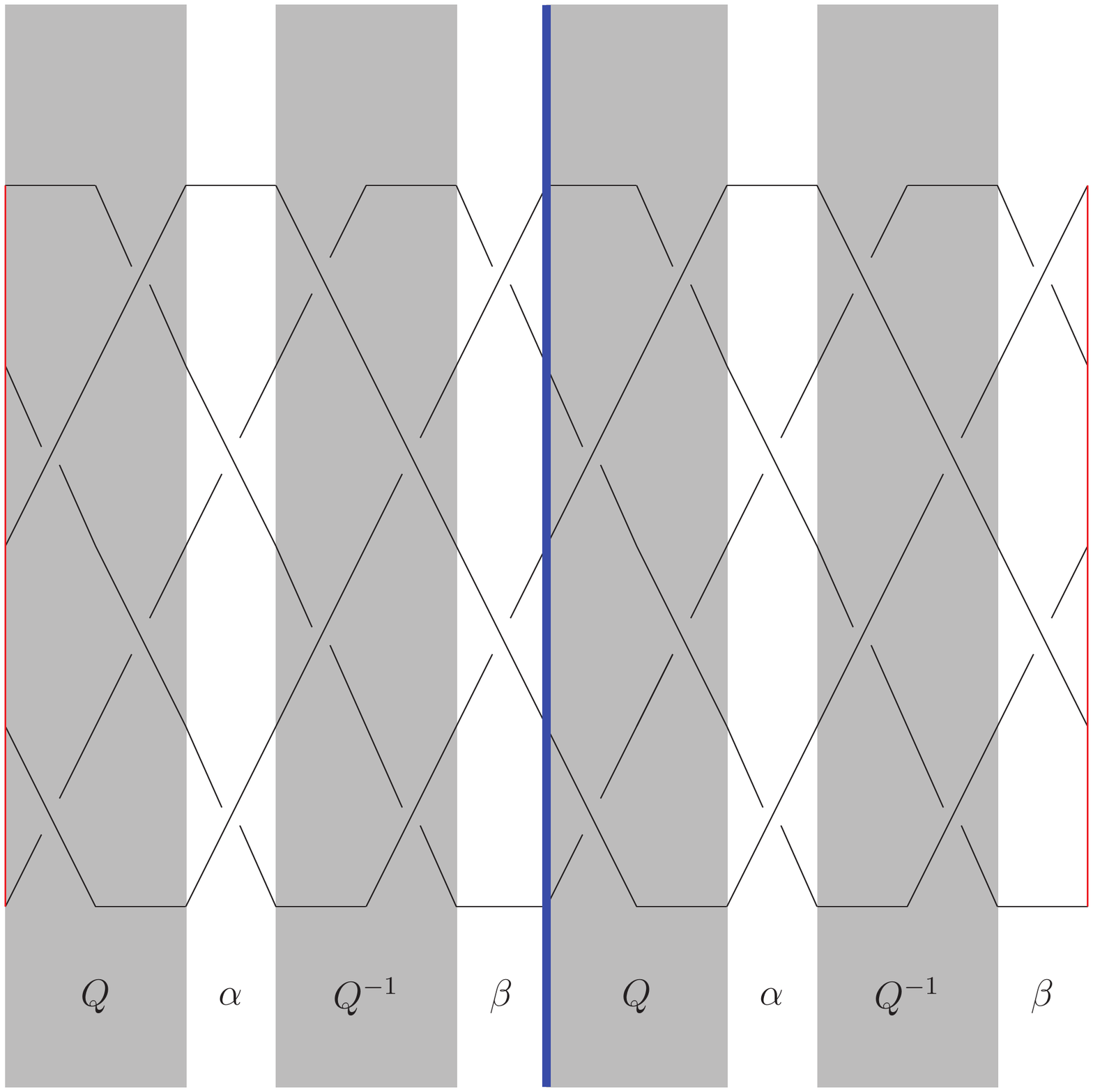} 
 
 \caption{ $B_{5,6,22} = B_{5,3,11}^2$, $\alpha=\sigma_2\sigma_4^{-1}$, $\beta=\sigma_1^{-1}\sigma_3$} 
\end{center} 
\end{figure} 
\clearpage
\section{Proof of the Main Theorem on the simple minimal braid $B_{N,q,p}$}\label{The structure of a simple minimal braid: proofs}

 \begin{figure}[!h]
 	\begin{center} 
 		\fbox{ \includegraphics[scale=.3]{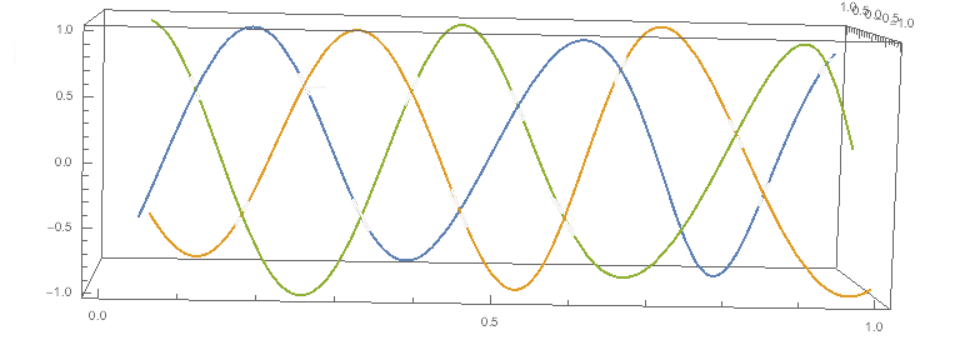}  
 		} 
 		\caption{Graph of $\{\psi_k\}_{,k=1,2,3}$ of knot   $K(3,7,5)$} 
 	\end{center} \end{figure}

 	 We recall some facts from [S-V]. We endow $\mathbb{R}^3$ with coordinates $(t,y,z)$:  the braid is the collections of the graphs in $\mathbb{R}^3$ of the functions $\psi_k$ for $k=1,...,N$: $$\psi_k=(\psi_k^{(1)},  \psi_k^{(2)}):[\eta,1+\eta]\longrightarrow\mathbb{R}^2$$ \begin{equation}\label{definition de la tresse} t\mapsto (y,z)=\big(\psi_k^{(1)}(t), \psi_k^{(2)}(t)\big)=\Big(\sin \frac{2\pi q}{N}(t+k), \cos \frac{2\pi p}{N}\big(t+k+\phi\big)\Big) \end{equation} 
 \begin{figure}[!h]\label{image de la tresse 357}
 	\begin{center} 
 		\fbox{
 			\includegraphics[scale=.8]{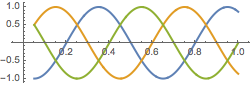}
 			\includegraphics[scale=.6]{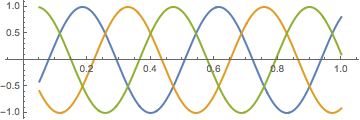} 
 		} 
 		\caption{Graph of $\{\psi_k^{(1)}\}_{k=1,2,3}$  and $\{\psi_k^{(2)}\}_{k=1,2,3}$ of knot   $K(3,7,5)$} 
 	\end{center} \end{figure}

\subsection{Periodic braids: proof of Proposition \ref{ce qui nous mene a l'assumption}}\label{preuve de la proposition sur les tresses periodiques}
We now prove Proposition \ref{ce qui nous mene a l'assumption} (stated in \ref{subsection: Contents of the paper}). We divide the interval $[\eta,1+\eta]$ into $d$ intervals $$I_n=[\frac{n}{d}+\eta,\frac{n+1}{d}+\eta],\ \ \ \ \ \ \ n=0,...,d-1.$$ After a change of variables $t\mapsto s=dt$, we see that the braid above an interval $I_n$ consists in the collection of graphs of the functions $[d\eta,1+d\eta]\longrightarrow \mathbb{R}^2$ \begin{equation} \label{apres le changement de variables} s\mapsto \Big(\sin \frac{2\pi \tilde{q}}{N}(s+dk), \cos \frac{2\pi \tilde{p}}{N}\big(s+dk+d\phi\big)\Big) \end{equation} Since $(N,d)=1$, the map $\ k\mapsto kd\ (mod\ N)\ $ induces a permutation of $\{1,...,N-1\}$; hence the piece of $B_{N,q,p}$ above $I_n$ is the collection of graphs above $[\eta,1+\eta]$ of the functions $$s\mapsto \Big(\sin \frac{2\pi \tilde{q}}{N}(s+k), \cos \frac{2\pi \tilde{p}}{N}\big(s+k+d\phi\big)\Big)$$ i.e. it is the braid $B_{N,\tilde{q},\tilde{p}}$, representing the knot $K(N,\tilde{q},\tilde{p},d\phi)$; this proves Proposition \ref{ce qui nous mene a l'assumption}.\ \ \ \qed\\ 
\subsection{Crossing values and crossing points of the braid}\label{la ou on mentionne les phases critiques}
The {\it braid shadow} is the projection of the braid onto the first two components $(t,y)$ of $\mathbb{R}^2$ i.e. the collection of the graphs of the $\psi_k^{(1)}$'s.\\ A {\it crossing point} $P$ of the braid  is the data of two different integers, $k$, $l$ with $0\leq k,l\leq N-1$ and a number $t\in [\eta, 1+\eta]$ called a {\it crossing value} such that $$\psi_k^{(1)}(t)=\psi_l^{(1)}(t)\ \ \ \ \  \mbox{i.e.}\ \ \ \ \ \ \sin \big(\frac{2\pi}{N}q(t+k)\big)=\sin \big(\frac{2\pi}{N}q(t+l)\big).$$ There is a  total of $(N-1)q$ crossing points, as in the case of the $(N,q)$ torus knot (where $q=p$).\\ A straightforward computation (cf. [S-V]) shows that, for a crossing point $P$ between the $k$-th and $l$-th strands of $B_{N,q,p}$, the corresponding crossing value $t$ verifies for some integer $m$ \begin{equation}\label{equ:valeur de t} t=-\frac{k+l}{2} + \frac{N}{4q}(2m+1) \end{equation}
 
 	 The sign $\Sigma(P)$ of a crossing point $P$ is
 	 \begin{figure}[!h]
 	 	\begin{center} 
			 \includegraphics[scale=.4]{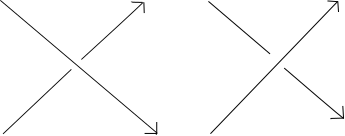} 
 	 		\caption{   crossing numbers   $+1 $ and  $-1$} 
 	 	\end{center} \end{figure}

\begin{equation}
\label{definition du signe d'un crossing point}
\Sigma(P)=\mbox{sign of}\ \ \ \big(\psi_k^{(2)}(t)-\psi_l^{(2)}(t)\big)\big(\psi_l^{(1)\prime}(t)-\psi_k^{(1)\prime}(t)\big)
\end{equation}

In [S-V], we computed this $\Sigma(P)$ as: \begin{equation}\label{formula pour le signe} \Sigma(P)=(-1)^{m}(-1)^{[p\frac{m}{q}+\frac{p}{2q}+ 	\frac{2p\phi}{N}]}(-1)^{[p\frac{k-l}{N}]}(-1)^{[q\frac{k-l}{N}]} \end{equation} where $[\ \ ]$ denotes the integral part. \subsubsection{Determination of the crossing points above a given crossing value} Let $t$ be a crossing value of $B_{N,q,p}$. We look for the $y$'s such that $(t,y)$ is a crossing point of the braid
shadow.\\ 
We derive from (\ref{equ:valeur de t}) the existence of at least one ordered pair of integers $(m,\mathfrak{s})$ such that \begin{equation}\label{les m et S pour une crossing value} t=-\frac{\frak{s}}{2}+\frac{N}{4q}(2m+1). \end{equation}\\  (with $\frak{s}=k+l$). There can be several $(m,\frak{s})$'s verifying (\ref{les m et S pour une crossing value}) for the same crossing value $t$; however \begin{equation}\label{inegalite pour s et m} 1\leq \frak{s}=k+l\leq 2N-3 \end{equation} which implies that there are at most two possible $(m,\frak{s})$'s  (Lemma \ref{nombre de m et S} below). We will see later  that for $t$ and $(m,\frak{s})$ given, a crossing point $(t,y)$ of the braid shadow above $t$ will be given by the data of $\frak{d}=k-l$.

\begin{lem}\label{nombre de m et S} If $t$ is a crossing value, one of the following two cases occurs:
	\begin{itemize}
		\item
	 {\bf 1st case}. There is exactly one ordered pair $(m,\frak{s})$, $1\leq \frak{s}\leq 2N-3$ verifying (\ref{les m et S pour une crossing value}); we denote it $(m(t),\frak{s}(t))$ and $\frak{s}(t)$ is either $N-2$, $N-1$ or $N$. 
	 \item {\bf 2nd case}. There exist exactly two  $(m,\frak{s})$'s satisfying (\ref{les m et S pour une crossing value}) with $1\leq \frak{s}\leq 2N-3$; we  denote them $(m(t),{\frak s}(t))$ and $(m(t)+q,{\frak s}(t)+N)$.
\end{itemize}
\end{lem}

\begin{proof} We let $(m(t), {\mathfrak s}(t))$ be the ordered pair such that ${\frak s}(t)$ is the smallest $\mathfrak{s}$ for the $(m,{\mathfrak s})$'s verifying  (\ref{les m et S pour une crossing value}) and (\ref{inegalite pour s et m}). If $(m_1,{\frak s}_1)$ and $(m_2,{\frak s}_2)$ both verify (\ref{les m et S pour une crossing value}) for the same $t$, we have \begin{equation} q({\frak s}_1-{\frak s}_2)=N(m_2-m_1) \end{equation} Since $q$ and $N$ are mutually prime, it follows that, for some integer $a$, \begin{equation}\label{difference entre les S et les m} \frak{s}_2=\frak{s}_1+aN\ \ \ \ m_2=m_1+aq \end{equation} 
	Since $\mathfrak{s}(t)$ is the smallest one, it verifies
	\begin{equation}
	\label{plus petit que N} \mathfrak{s}(t)\leq N.
	\end{equation}
	If we are in the 1st case of the Lemma \ref{nombre de m et S}, i.e. a single $(m, \frak{s})$, we derive from (\ref{difference entre les S et les m}) that ${\frak s}(t)+N$ does not verify (\ref{inegalite pour s et m}), i.e. \begin{equation}\label{S en dehors de l'intervalle} {\frak s}(t)+N>2N-3 \end{equation} Putting together (\ref{plus petit que N}) and (\ref{S en dehors de l'intervalle}), we get $$N-2\leq {\frak s}(t)\leq N$$ which concludes the proof of the 1st case.\\ The 2nd case is clear: since ${\frak s}(t)+2N>2N-3$, ${\frak s}(t)+N$ is the only other integer $\frak{s}$ in $[1, 2N-3]$ which can appear in (\ref{les m et S pour une crossing value}); the corresponding $m$ is $m(t)+q$. \end{proof} 
Lemma \ref{nombre de m et S} told us which $\mathfrak{s}$'s and $m$ occur for crossing points $(t,y)$ above a crossing value $t$: we now find the $k,l$'s such that $\mathfrak{s}=k+l$ and derive the $\sigma_i^{\pm}$'s corresponding to the $(t,y)$'s.
\begin{defn} Let $P=(t,y)$ be a crossing point; we denote by $i(P)\in\{1,...,N-1\}$  the corresponding generator subscript, i.e. $P$ is represented by $\sigma_{i(P)}$ or $\sigma_{i(P)}^{-1}$. \end{defn}
\begin{lem}\label{les valeurs de y} Let $t$ be a crossing value of $B_{N,q,p}$. \begin{enumerate} \item The point $P=(t,y)$ is a crossing point of the braid shadow if and only if \begin{equation}\label{valeur de y} y=(-1)^{m(t)}\cos(\frac{q\frak{d}}{N}\pi) \end{equation}
	 where $\frak{d}$ is any integer in $[1,...,N-1]$ of the same parity as $\frak{s}(t)$
	 \item
	 To determine $i(P)$, we  do the Euclidean division of $q\frak{d}$ by $2N$ \begin{equation}\label{Euclidean division} q\frak{d}=2Nn+w \end{equation} with $n\geq 0$, $-N<w<N$.
	 	\begin{enumerate}
	 		\item
	If $m(t)$ is even, 
	\begin{equation}
	i(P)=i\Big(t,\cos(\frac{q\frak{d}}{N}\pi)\Big)=|w|
	\end{equation}
		\item
		If $m(t)$ is odd, 
		\begin{equation}
		i(P)=i\Big(t,-\cos(\frac{q\frak{d}}{N}\pi)\Big)=N-|w|
		\end{equation}
		\end{enumerate}

		  \end{enumerate} 
		  \end{lem} 
		  \begin{proof} {\bf Proof of 1}: we treat separately the two cases of Lemma \ref{nombre de m et S}. \begin{enumerate} \item

{\bf 1st case}: a single ordered pair $(m,\frak{s})$. \\ Let $k,l$ such that $k+l=\frak{s}(t)$ and assume that $l<k$; note that $\frak{d}=k-l$ has the parity of $\frak{s}(t)$. \begin{enumerate} \item If ${\frak s}(t)=N-1$, the smallest possible value for $l$ is $0$ and $k-l$ runs through all integers $\frak{d}$, $1\leq \frak{d}\leq N-1$ with the parity of ${\frak s}(t)=N-1$. \item If ${\mathfrak s}(t)=N-2$ (resp. ${\mathfrak s}(t)=N$), the smallest value for $l$ is $0$ (resp. $1$) and $\mathfrak{d}$ runs through the integers in $[1, N-2]$ with the parity of $N-2$ or $N$; since $N-1$ has parity opposite to $N$ and $N-2$, we can actually assume  $\frak{d}$ in $[1, N-1]$. \end{enumerate} To derive (\ref{valeur de y}), we plug (\ref{les m et S pour une crossing value}) into \begin{equation}\label{l'expression de y} y=\sin\big(\frac{2\pi}{N}q(t+k)\big) \end{equation} \item {\bf 2nd case}: two ordered pairs: $(m(t),\frak{s}(t))$ and $(m(t)+q,\frak{s}(t)+N)$. \begin{enumerate} \item We first consider the $k,l$'s such that $l<k$ and $k+l={\mathfrak s}(t)$. As above,  $k-l$ runs through the integers $\frak{d}$  with the parity of $ {\mathfrak s}(t)$ and such that \begin{equation}\label{premiere inegalite sur frakd} 1\leq \frak{d}\leq {\mathfrak s}(t) \end{equation} \item 
	If $k+l={\mathfrak s}(t)+N$, we look at $ {\mathfrak d}=k-l$'s with $l<k$: $$l= {\mathfrak s}(t)+N-k\geq {\mathfrak s}(t)+N-(N-1)={\mathfrak s}(t)+1$$ $$\mbox{hence\  }{\mathfrak d}={\mathfrak s}(t)+N-2l\leq {\mathfrak s}(t)+N-2{\mathfrak s}(t)-2=N-{\mathfrak s}(t)-2{\mbox{\ and}}$$ 
	\begin{equation}
	\label{intervalle possible pour delta}
	 1\leq \frak{d}\leq N-{\mathfrak s}(t)-2 \end{equation}
	 Moreover every integer in  $\frak{d}\in [1,  N-{\mathfrak s}(t)-2]$
	 with the parity of ${\mathfrak s}(t)+N$ is a legitimate $\frak{d}$, i.e. there exist $k,l$ in $\{1,...,N-1\}$ with   $\frak{d}=k-l$ and $k+l={\mathfrak s}(t)+N$; for example
	$N-{\mathfrak s}(t)-2=(N-1)-({\mathfrak s}(t)+1)$. \\
	\\
	 Using (\ref{l'expression de y}), we derive the $y$-coordinate of the crossing point: \begin{equation}\label{formule pour y pour m plus q; deux} y=(-1)^{m(t)}(-1)^q\cos(\pi\frac{q\frak{d}}{N}) =(-1)^{m(t)}\cos(\pi\frac{q\tilde{\frak{d}}}{N}) \end{equation} where $\tilde{\frak{d}}=N-\frak{d}$; if $\frak{d}$ verifies (\ref{intervalle possible pour delta}), then \begin{equation} \label{intervalle pour d tilda} {\mathfrak s}(t)+2\leq \tilde{\frak{d}}\leq N-1 \end{equation} 
	 Since $\frak{d}$ has the parity of ${\mathfrak s}(t)+N$,  $\tilde{\frak{d}}$ has the parity of ${\mathfrak s}(t)$.

\end{enumerate} Putting together the intervals (\ref{premiere inegalite sur frakd}) and (\ref{intervalle pour d tilda}) concludes the proof of the 2nd case. \end{enumerate}

{\bf Proof of 2}. We derive from (\ref{Euclidean division}) that $$\cos(\pi\frac{q\frak{d}}{N})=\cos(\pi\frac{w}{N}) =\cos(\pi\frac{|w|}{N})$$ so 2 (a) of Lemma \ref{les valeurs de y}  follows from the fact that the function ${\it cos}$ is decreasing on $(0,\pi)$: $$\cos \frac{\pi}{N}>\cos \frac{2\pi}{N}>...>\cos \frac{(N-1)\pi}{N}.$$
 
 and 2 (b) of Lemma \ref{les valeurs de y} follows from $$-\cos(\pi\frac{q\frak{d}}{N})=-\cos(\pi\frac{w}{N})= -\cos(\pi\frac{|w|}{N})=\cos(\pi\frac{N-|w|}{N}).$$
 \end{proof}
 To see how many crossing points $(t,y)$ occur above $t$, i.e. how many values (\ref{valeur de y}) takes for a given $t$, we notice the following. \begin{itemize}
 	\item
 	If $u,v\in\{1,...,N-1\}$ and $\cos(\pi\frac{qu}{N})=\cos(\pi\frac{qv}{N})$, then $u=v$.  \item If $u,v\in\{1,...,N-1\}$ and $\cos (\pi\frac{qu}{N})=-\cos (\pi\frac{qv}{N})$, then $u=N-v$.
 \end{itemize} So the sets $\{\cos(\pi\frac{qu}{N})\slash 1\leq u\leq N-1\}$ and $\{-\cos(\pi\frac{qu}{N})\slash 1\leq u\leq N-1\}$ are identical. We derive from Lemma  \ref{les valeurs de y}
 \begin{cor}\label{differentes parites des generateurs de tresses} Let $t$ be a crossing value. The indices $i(P)$'s for the crossing points $P$'s above $t$ are all the $i$'s in $\{1,...,N-1\}$ and with  
 	
 	\begin{enumerate} \item the parity of ${\mathfrak d}$ and $\mathfrak{s}(t)$ if $m(t)$ is even \item the parity of ${\mathfrak d}+N$ and $\mathfrak{s}(t)+N$ if $m(t)$ is odd. \end{enumerate} \end{cor} 
 \subsubsection{The sign of the crossing points}
 We now compute the sign of the crossing points described in Lemma \ref{les valeurs de y}.
 \begin{lem}

   The sign of the crossing point corresponding to the $(t,y)$ appearing in (\ref{valeur de y}) is given by \begin{equation}\label{signe de y en fonction de u}
  \Sigma(t,y)= (-1)^{m(t)}\sigma\big(p\frac{m(t)}{q}+\frac{p}{2q}+ \frac{2p\phi}{N}\big)\sigma(\frac{q\frak{d}}{N})\sigma(\frac{p\frak{d}}{N}) \end{equation} 
  where $\sigma(r)$ is the parity of the integer part:
  \begin{equation}\sigma(r)=(-1)^{[r]}
  \end{equation}
 
\end{lem}
 \begin{proof} We recall (Lemma \ref{nombre de m et S}) that 
 	$k+l=\frak{s}(t)$ or $k+l=\frak{s}(t)+N$.\\ If $k+l=\frak{s}(t)$ and $\frak{d}=k-l$, then (\ref{signe de y en fonction de u}) is just the formula (\ref{formula pour le signe}) for the sign of a crossing point. So we assume that  $k+l={\mathfrak s}(t)+N$: we have seen above that  $y$ is given by (\ref{formule pour y pour m plus q; deux}) with  $k-l=N-\tilde{\frak{d}}$ and we write (\ref{formula pour le signe}) for the sign of the crossing point $$(-1)^{m(t)+q}\sigma\left(q(\frac{N-\tilde{\frak{d}}}{N})\right)\sigma\left(p(\frac{N-\tilde{\frak{d}}}{N})\right)\sigma\Big(p\frac{m(t)+q}{q}+\frac{p}{2q}+ \frac{2p\phi}{N}\Big)$$ $$=(-1)^{m(t)}\sigma(q\frac{\tilde{\frak{d}}}{N})\sigma(p\frac{\tilde{\frak{d}}}{N})\sigma\Big(p\frac{m(t)}{q}+\frac{p}{2q}+ \frac{2p\phi}{N}\Big)$$ \end{proof}
 
 We write the last two factors of (\ref{signe de y en fonction de u}) in terms of $i(P)$:
  \begin{lem}\label{les deux derniers facteurs du signe} We let $P=\big(t,(-1)^{m(t)}\cos\pi\frac{q\frak{d}}{N}\big)$ be a crossing point of $B_{N,q,p}$ with $m(t)$,  $\frak{d}$ and  $w$ as in 2) of Lemma \ref{les valeurs de y}. Then 
  	\begin{equation} \label{la ou epsilon va peut-etre disparaitre} \sigma(\frac{q\frak{d}}{N})\sigma(\frac{p\frak{d}}{N})=
  	\left\{ \begin{array}{c}
  	\sigma\big(\frac{Bpi(P)}{N}\big)\mbox{\ \ if\ }p\mbox{\ is odd}\\ 
  	(-1)^{m(t)}\sigma\big(\frac{Bpi(P)}{N}\big)\mbox{\ \ if\ }p\mbox{\ is even}\end{array} \right. 
  	\end{equation} 
  	
  \end{lem} 
  \begin{proof}
  	It follows from (\ref{Euclidean division}) that 
  	\begin{equation}
  	\label{premiere identite avec sigma}
  	\sigma(\frac{q\frak{d}}{N})=\sigma(\frac{w}{N}).
  	\end{equation}
  	We recall (\ref{definition des A et B}), namely $2NA+Bq=1$, hence $\frak{d}=2NA\frak{d}+Bq\frak{d}$; putting this together with the Euclidean division in (\ref{Euclidean division}), we have
  	\begin{equation}\label{une equation avec frak de d} 
  	\frak{d}=2NA\frak{d}+2NnB+Bw 
  	\end{equation} 
  	Thus
  	\begin{equation}
  	\label{deuxieme identite avec sigma}
  	\sigma(\frac{p\frak{d}}{N})=\sigma(\frac{pBw}{N}).
  	\end{equation} Now
  	\begin{equation}\label{extremement intermediaire} 
  	\sigma(\frac{w}{N})\sigma(\frac{pBw}{N})=
  	\sigma(\frac{|w|}{N})\sigma(\frac{pB|w|}{N})=\sigma(\frac{pB|w|}{N})
  	\end{equation} 
  	\begin{itemize} \item If $m(t)$ is even, then $i(P)=|w|$ and the Lemma is proved. \item If $m(t)$ is odd, we use the fact that $B$ is odd to write $$\sigma(\frac{pB|w|}{N})=\sigma\big(\frac{pB(N-i(P))|w|}{N}\big)
  		=
  		\left\{ \begin{array}{c}
  		\sigma(\frac{pBi(P)}{N})\mbox{\ \ if\ }p\mbox{\ is odd}\\ 
  		-\sigma(\frac{pBi(P)}{N})\mbox{\ \ if\ }p\mbox{\ is even}\end{array} \right.  $$ 
  	\end{itemize}
  \end{proof}
  It follows from Lemma \ref{les deux derniers facteurs du signe} that  the sign $\Sigma(P)$ given in (\ref{signe de y en fonction de u}) of a crossing point $P$ of crossing value $t$ is ($\epsilon$ has been defined in (\ref{definition d'epsilon avec le signe de Conway}) above)
  \begin{itemize}
  	\item  $\sigma\Big(p\frac{m(t)}{q}+\frac{p}{2q}+ \frac{2p\phi}{N}\Big)\epsilon(N,q,p)(i)$ if $p$ is even
  	\item $\sigma\Big(p\frac{m(t)}{q}+\frac{p}{2q}+ \frac{2p\phi}{N}\Big)(-1)^{m(t)}\epsilon(N,q,p)(i)$ if $p$ is odd. 
  \end{itemize} 
  We recall (see the formulae (\ref{deuxieme apparition de alpha et beta}) in the Main Theorem) that the $\pm 1$-exponent of a $\sigma_i$ in $\alpha_{N,q,p}$ or $\beta_{N,q,p}$ is $\epsilon(N,q,p)(i)$, hence
  \begin{lem}\label{exposant de alpha et beta} 
  	The $\pm 1$-exponent of the $\alpha_{N,q,p}$ or $\beta_{N,q,p}$ corresponding to a crossing value $t$ in $B_{N,q,p}$ is
  	
  	\begin{enumerate} \item  $\sigma\Big(p\frac{m(t)}{q}+\frac{p}{2q}+ \frac{2p\phi}{N}\Big)$ if $p$ is even
  		\item   $(-1)^{m(k)}\sigma\Big(p\frac{m(t)}{q}+\frac{p}{2q}+ \frac{2p\phi}{N}\Big)$ if $p$ is odd.
  	\end{enumerate}
  \end{lem}
 The formulae in Lemma \ref{exposant de alpha et beta} depend on $m(t)$, where $t$ goes through the $2q$ crossing values. If $t$ is the $h$-th crossing value, for $h=1,...,2q$, we want to have $m(t)$ directly as an expression in $h$ so we
 number the crossing values \begin{equation}\label{liste des t} t_1<t_2<...<t_q<...<t_{2q} \end{equation} and for any integer $k$, with $1\leq k\leq 2q$, we let \begin{equation}\label{m et s pour un des t} m(k)=m(t_k)\ \ \ \ \ \ \ \frak{s}(k)=\frak{s}(t_k) \end{equation}
If $t_k$ and $t_{k+1}$ are two consecutive crossing values, we derive from (\ref{les m et S pour une crossing value}) \begin{equation}\label{crossing values qui se suivent} t_{k+1}-t_k=\frac{1}{2q}\Big[q\big(\frak{s}(k)-\frak{s}(k+1)\big)+N\big(m(k+1)-m(k)\big)\Big]\geq \frac{1}{2q} 
\end{equation} 
Since there are $2q$ crossing values in $[\eta,1+\eta]$, (\ref{crossing values qui se suivent}) is an equality and we have 
\begin{equation}\label{relation entre les S et m pour les valeurs qui se suivent} q\big(\frak{s}(k)-\frak{s}(k+1)\big)+N\big(m(k+1)-m(k)\big)= 1 \end{equation} We note in passing that, if we plug (\ref{relation entre les S et m pour les valeurs qui se suivent}) into Proposition \ref{differentes parites des generateurs de tresses}, we get the confirmation of the obvious fact
\begin{lem} The crossing points above $t_k$ and $t_{k+1}$ are represented by $\sigma_i^{\pm 1}$'s with $i$'s of opposite parities. 
\end{lem}

 We now confront (\ref{relation entre les S et m pour les valeurs qui se suivent}) with $2NA+Bq=1$ and derive the existence of an integer $\nu_k$ such that $$m(k+1)-m(k)=\nu_k q+ 2A.$$ Thus, for any $k$, there exists an integer $a_k$ such that \begin{equation}\label{une expression pour mk} m(k)=m(1)+a_kq+2(k-1)A \end{equation}  \begin{equation}\label{moins un puissance m} (-1)^{m(k)}=(-1)^{m(1)}(-1)^{a_k} \end{equation}

Define $\phi_0$ by \begin{equation}\label{definition de la phase qui tue} p\frac{m(1)}{q}-\frac{2Ap}{q}+\frac{p}{2q}+\frac{2p\phi_0}{N}=0 \end{equation}
We see (\ref{formula pour le signe}) that $\phi_0$ is a critical phase, i.e. a phase for which the knot $K(N,p,q,\phi)$ is singular. So we pick a phase
	\begin{equation}
	\label{definition de la phase phi}
	\phi=\phi_0+\xi
	\end{equation}
	 where $\xi$ is a very small positive number. Using (\ref{une expression pour mk}), we rewrite \begin{equation}\label{expression pour s avec xi} \sigma\Big(p\frac{m(t)}{q}+\frac{p}{2q}+ \frac{2p\phi}{N}\Big)=(-1)^{a_kp}(-1)^{\big[\frac{2Ap}{q}k+\xi\big]} \end{equation} It follows from Lemma \ref{exposant de alpha et beta} and equations (\ref{une expression pour mk}), (\ref{expression pour s avec xi}) that the exponent of the $\alpha_{N,q,p}$ or $\beta_{N,q,p}$ at the $k$-th crossing value is 
$$
\left\{ \begin{array}{c}
(-1)^{m(1)}(-1)^{\big[\frac{2Ap}{q}k+\xi\big]}\mbox{\ \ if\ }p\mbox{\ is odd}\\
\\ 
(-1)^{\big[\frac{2Ap}{q}k+\xi\big]}\mbox{\ \ if\ }p\mbox{\ is even} \end{array} \right.
$$
 Since we are working up to mirror transformation, we assume  $$(-1)^{m(1)}=1.$$ 
 We  now conclude: the expression $[\frac{2Ap}{q}k+\xi]$ is equal to $1$ for $k=q,2q$ and equal to $\lambda(k)$ for the other $k$'s. Going back to the statement of the Main Theorem, this gives us the exponent for the $k$-th crossing values with $k\leq q$ or $k=2q$. We settle the case of the $k$'s with $q<k<2q$ by noticing that
 $$\lambda(2q-k)=-\lambda(k).$$ \qed
 \section{The four-genus}\label{section sur le four genus} 
\subsection{Ribbon knots}\label{preliminaires sur le ribbon}
A {\it ribbon knot}  in $\mathbb{S}^3$ bounds a disk in $\mathbb{S}^3$ with only ribbon singularities; equivalently it bounds an embedded disk in $\mathbb{B}^4$ with does not have local maxima for the distance to the origin of $\mathbb{B}^4$. Thus a ribbon knot is slice, i.e. its $4$-genus is zero; the long-standing Slice-Ribbon conjecture asks if the converse is true. 
\subsubsection{Proof of Theorem \ref{theoreme sur le ribbon}}
Th. \ref{theoreme sur le ribbon} has been stated by Lamm and also follows from his more general construction of ribbon symmetric unions ([La 2], [K-T]). His proof is fairly allusive so we felt it would be useful to give a more detailed proof.\\
We recall the well-known fact:
\begin{prop}\label{noeud symetrique est ribbon} Let $K$ be a knot in $\mathbb{R}^3$ which is symmetric with respect to a plane $P$ in $\mathbb{R}^3$; then it is ribbon. 
\end{prop}

\begin{proof}
	We endow $\mathbb{R}^3$ with the frame $Oxyz$ and assume that $P$ is defined by the equation $x=0$. By genericity arguments, we assume
	\begin{enumerate}
		\item
		$K$ meets $P$ at a finite number of points
		\item 
		outside of $P$, $K$ is never tangent to the direction of $Ox$. 	
	\end{enumerate}
	Since $K$ has one component, it meets $P$ at exactly two points.\\ 
	We let $K_+$ (resp. $K_-$) be the intersection of $K$ with the half-space of $\mathbb{R}^3$ defined by $z\geq 0$ (resp. $z\leq 0$): $K_+$ and $K_-$ are both diffeomorphic to a closed interval.\\
	Letting $S$ be the symmetry in $\mathbb{R}^3$ with respect to $P$, we let
	$$\Phi=[0,1]\times K_+\longrightarrow\mathbb{R}^3$$
	\begin{equation}
	(t,X)\mapsto tX+(1-t)S(X)
	\end{equation}
	The self-intersections of $\Phi$ are given by the data of  $t_1,t_2, X_1, X_2$ such that
	\begin{equation}\label{egalite}
	t_1X_1+(1-t_1)S(X_1)=t_2X_2+(1-t_2)S(X_2)
	\end{equation} 
	We denote by $(x_i,y_i,z_i)$, $i=1,2$ the coordinates of $X_i$. Since $S(x_i,y_i,z_i)=(-x_i,y_i,z_i)$, (\ref{egalite}) implies that 
	$$y_1=y_2, z_1=z_2.$$
	Thus the line segments $I_1=X_1S(X_1)$ and $I_2=X_2S(X_2)$ are both included in the line which is defined by the equations $y=y_1,z=z_1$. Moreover, one of them is included in the other one and we have a ribbon singularity.

\end{proof}
The Main Theorem tells us that, if $p$ and $q$ are mutually prime, with $q$ odd, $B_{N,q,p}$ is as in Fig. \ref{premier dessin}.\\
\begin{figure} [h]
\adjustbox{width = 4in, height = 2in}{ 
	\begin{tikzpicture}
	\draw  (6,0) arc [radius=4, start angle=90, end angle=270];
	\node at (3,-4) {$Q$};
	\draw  (6,-2) arc [radius=2, start angle=90, end angle=270];
	\draw (6,0)--(6,-2);
	\draw (6,-6)--(6,-8); 
	\draw  (8,-8) arc [radius=4, start angle=-90, end angle=90];
	\node at (3,-4) {$Q$};
	\draw  (8,-6) arc [radius=2, start angle=-90, end angle=90];
	\draw (8,0)--(8,-2);
	\draw (8,-6)--(8,-8);
	\node at (11,-4) {$Q^{-1}$};
	\draw (6,0)--(8,-0.5);
	\draw (6,-0.5)--(6.8,-0.3); \draw (7.2,-0.2)--(8,0);
	\draw (6,-1)--(8,-1.5);
	\draw (6,-1.5)--(6.8,-1.3); \draw (7.2,-1.2)--(8,-1);
	\draw (6,-1.5)--(6.8,-1.3); \draw (7.2,-1.2)--(8,-1);
	\draw (6,-7)--(8,-6.5);
	\draw (6,-7.5)--(8,-8);
	\draw (6,-6.5)--(6.8,-6.7); \draw (7.2,-6.8)--(8,-7);
	\draw (6,-8)--(6.8,-7.8); \draw (7.2,-7.7)--(8,-7.5);
	\end{tikzpicture} \hskip 1 in 
	\begin{tikzpicture}
	\draw  (6,0) arc [radius=4, start angle=90, end angle=270];
	\node at (3,-4) {$Q$};
	\draw  (6,-2) arc [radius=2, start angle=90, end angle=270];
	\draw (6,0)--(6,-2);
	\draw (6,-6)--(6,-8); 
	\draw  (8,-8) arc [radius=4, start angle=-90, end angle=90];
	\node at (3,-4) {$Q$};
	\draw  (8,-6) arc [radius=2, start angle=-90, end angle=90];
	\draw (8,0)--(8,-2);
	\draw (8,-6)--(8,-8);
	\node at (11,-4) {$Q^{-1}$};
	\draw (6,0)--(8,0);
	\draw (6,-0.5)--(8,-0.5);
	\draw (6,-1)--(8,-1);
	\draw (6,-1.5)--(8,-1.5);
	\draw (6,-6.5)--(8,-6.5);
	\draw (6,-7)--(8,-7);
	\draw (6,-7.5)--(8,-7.5);
	\draw (6,-8)--(8,-8);
	
	\end{tikzpicture}

	}
	\caption{$B(5,q,p)$ and  the  link $L$ }\label{premier dessin}\label{deuxieme dessin}
\end{figure}
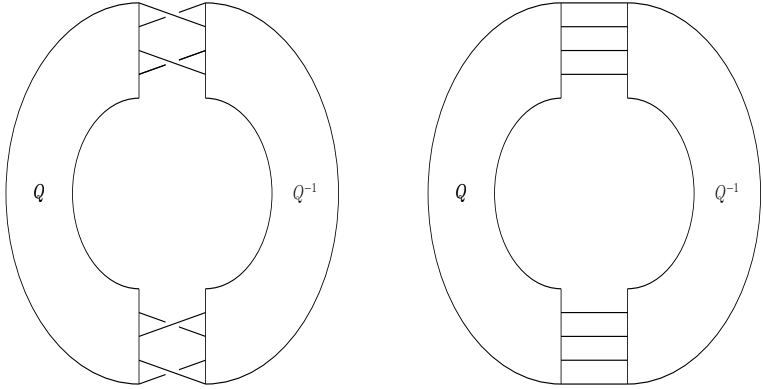
If $N$ is the number of strands, there are $N-1$ half-twist tangles connecting $Q$ and $Q^{-1}$; we replace them by $N-1$ tangles and get the $N$-component link $L$ of Fig. \ref{deuxieme dessin} which is symmetric w.r.t. a plane.
 
Proposition \ref{noeud symetrique est ribbon} tells us that $L$ bounds $N$ ribbon disks $D_1$, ..., $D_N$; and the same arguments show us that two of these disks only have ribbon-type intersection.\\
We now connect each $D_i$ to $D_{i+1}$ by a half-twisted band bounded by the half-twist tangle of Fig. \ref{premier dessin}. The resulting surface is a topological disk with only ribbon singularities. \qed

\subsection{General case: proof of Theorem \ref{theoreme sur le four-genus}}
We use an idea by Brandenbursky and Kedra ([B-K]). If $b$ is a $N$-braid, we denote by $\hat{b}$ the link obtained by closing the braid $b$. If $b_1$ and $b_2$ are $N$-two braids, [B-K] constructed a cobordism of Euler characteristic $-N$ between 
the closure of the product $\widehat{b_1b_2}$ and the disjoint union of the closures $\widehat{b_1}\sqcup\widehat{b_2}$.\\
Letting $\tilde{q}=\frac{q}{N}$, $\tilde{p}=\frac{p}{N}$, we recall that 
$$B_{N,q,p}=B_{N,\tilde{q},\tilde{p}}^d.$$
Applying [B-K]'s result $d$ times, we derive a cobordism in $\mathbb{B}^4$ of Euler characteristic $-N(d-1)$ between  $\widehat{B_{N,q,p}}$ and $\underbrace{B_{N,\tilde{q},\tilde{p}}\sqcup B_{N,\tilde{q},\tilde{p}}\sqcup...B_{N,\tilde{q},\tilde{p}}}_{d\ \mbox{copies}}$. Since $\widehat{B_{N,\tilde{q},\tilde{p}}}$ is a ribbon knot (Theorem \ref{theoreme sur le ribbon}), it bounds an embedded disk in $\mathbb{B}^4$. Thus $\widehat{B_{N,q,p}}$ bounds a surface of Euler characteristic
$$-d(N-1)+d=1-(d-1)(N-1).$$
We recover the formula (\ref{inegalite pour le four genus}) for the genus and Theorem \ref{theoreme sur le four-genus} is proved. 
\qed

\subsection{Quasipositive knots: proof of Proposition \ref{four genus; egalite}}
 Lee Rudolph (see [Ru] for details) defines a braid $\gamma\in{\bf B}_N$ to be {\it quasipositive} if it is a product of conjugates $w\sigma_i w^{-1}$ of positive braid generators, i.e. \begin{equation}\label{quasipositive} \gamma=w_1\sigma_{i_1} w_1^{-1}w_2\sigma_{i_2} w_2^{-1}...w_k\sigma_{i_k} w_k^{-1} \end{equation} \begin{thm}\label{rudolph}([Ru]) If $\gamma$ is a closed quasipositive braid written as in (\ref{quasipositive}) closing in a knot $\hat{\gamma}$, its four-genus verifies $$1-2g_4(\hat{\gamma})=n-k.$$ \end{thm} It is easy to check that under the assumptions of Proposition \ref{four genus; egalite}, the exponents of all the $\sigma_{2k}$'s and $\sigma_{2k+1}$'s appearing respectively in $\alpha_{N,q,p}$ and $\beta_{N,q,p}$ are all of the same sign. Since we are working up to mirror symmetry, we can assume all these exponents to be equal to $1$; thus \begin{equation}\label{tresses positives} \alpha_{N,q,p}=\prod_{2\leq 2k\leq N-1}\sigma_{2k}\ \ \ \ \ \ \beta_{N,q,p}=\prod_{1\leq 2k+1\leq N-1}\sigma_{2k+1} \end{equation} 
	Hence $B_{N,p,q}=(Q\alpha_{N,q,p}Q^{-1}\beta_{N,q,p})^d$ is a quasipositive braid; Theorem \ref{rudolph} tells us that $1-2g_4(K(N,q,p))=N-d(N-1)$ and Proposition \ref{four genus; egalite} follows.

\section{Trivial knots: proof of Proposition \ref{liste de noeuds triviaux}}\label{la section sur les noeuds triviaux}
\subsection{The knot $K(N,q,q+N)$ is trivial}

We set
\begin{equation}
A=\prod_{1\leq 2k\leq N}\sigma_{2k}\ \ \ B=\prod_{1\leq 2k+1\leq N}\sigma_{2k+1}
\end{equation}

	\begin{lem}\label{forme de la tresse triviale}
		\begin{equation}\label{tresse triviale}
		B_{N,q,q+N}=A(BA)^{\frac{q-1}{2}}
		(B^{-1}A^{-1})^{\frac{q-1}{2}}B^{-1}
		\end{equation}
	\end{lem}
	The main thing to note about this formula is that all the positive generators are on one side and all the negative generators are on the other side.
	\begin{proof}
		We use the Main Theorem. It assumes that $q$ is odd but  in the present case, if $q$  is even, $N$ has to be odd, hence $q+N$ is odd and we switch $q$ and $q+N$ to apply the theorem. We compute
		 \begin{equation}
		\epsilon(i)=(-1)^i\ \ \ \ \lambda(k)=(-1)^k
		\end{equation}
		Thus $\alpha_{N,q,q+N}=\prod_{1\leq 2k\leq N}\sigma_{2k}=A$ and  $\beta_{N,q,q+N}=\prod_{1\leq 2k+1\leq N}\sigma^{-1}_{2k+1}=B^{-1}$. We conclude by noticing that
		$$(AB)^{\frac{q-1}{2}}A=A
		(BA)^{\frac{q-1}{2}}$$
		\end{proof}
	We construct a trivial pure braid $\mathcal{B}_N$; we will show that
	 $B_{N,N+q,q}$ is the product of a power of 
	$\mathcal{B}_N$  and of a piece of $\mathcal{B}_{N}$.
	
			If $N=2k$ is even, we let
			\begin{equation}
			{\mathcal B}_{N}=(BA)^k(B^{-1}A^{-1})^k.
			\end{equation}
			
			If $N=2k+1$ is odd, we let 
			\begin{equation}
			{\mathcal B}_N=(BA)^kBA^{-1}(B^{-1}A^{-1})^k.
			\end{equation}
In both cases, we check that the corresponding permutation between the endpoints of the braid is the identity, thus ${\mathcal B}_N$ is a pure braid.\\
\\
To prove that it is a trivial braid, we discuss when one strand of ${\mathcal B}_N$ is above another one; so let us fix some terminology.\\

	We number the strands of $\mathcal{B}_N$: the $j$-th strand, $0\leq j\leq N-1$ is the strand starting at the $(j+1)$-th point on the left (the points being counted from top to bottom). \\
	We say that the $j$-th strand is above the $k$-th strand if, wherever there is a crossing point between these two strands, the $j$-th strand is above the $k$-th strand. As an exemple, in Fig. \ref{second dessin}, the red strand is above all the other strands. 
	\begin{lem}\label{au dessus de toutes les suivantes}
		\label{brins au dessus}

		If $j,k$ are two integers with $0\leq j<k\leq N-1$, the $j$-th strand of ${\mathcal B}_N$ is above the $k$-th strand. 
		Thus ${\mathcal B}_N$ closes in $N$ unlinked trivial links, i.e. ${\mathcal B}_N=1$.

	\end{lem}
	The figure \ref{second dessin} illustrates the lemma.
	
	\begin{figure}[h] 
	\adjustbox{height = 1.5in, width = 3 in}{  
		\begin{tikzpicture}  \hskip 2 in 
		
		\draw[color=red] (0,0)--(6,-6)--(8,-6)--(14,0)--(16,0);
		\draw[color=green] (0,-2)--(0.75,-1.25);
		\draw[color=green] (1.25,-0.75)--(2,0)--(4,0)--(8.6,-4.6);
		\draw[color=green] (9.2,-5.2)--(10,-6)--(12,-6)--(16,-2);
		\draw[color=blue] (0,-4)--(2,-6)--(4,-6)--(4.75,-5.25);
		\draw[color=blue] (5.25,-4.75)--(6.74,-3.25);
		\draw[color=blue] (7.25,-2.75)--(10,0)--(12,0)--(12.8,-0.8);
		\draw[color=blue] (13.2,-1.2)--(14.8,-2.7);
		\draw[color=blue] (15.2,-3.1)--(16,-4);
		\draw (0,-6)--(0.75,-5.25);
		\draw (0,-6)--(0.75,-5.25);
		\draw (1.25,-4.75)--(2.75,-3.25);
		\draw (3.25,-2.75)--(4.75,-1.25);
		\draw (5.25,-0.75)--(6,0)--(8,0)--(8.75,-0.75);
		\draw (9.25,-1.25)--(10.75,-2.75);
		\draw (11.25,-3.25)--(12.75,-4.75);
		\draw (13.25,-5.25)--(14,-6)--(16,-6);
		\draw[very thick] (8,1)--(8,-7);
		\end{tikzpicture}
		}
		\caption{$B_{4,4}=\sigma_1\sigma_3\sigma_2\sigma_1\sigma_3\sigma_2\sigma^{-1}_1\sigma^{-1}_3\sigma^{-1}_2\sigma^{-1}_1\sigma^{-1}_3\sigma^{-1}_2$}\label{second dessin}
	\end{figure}
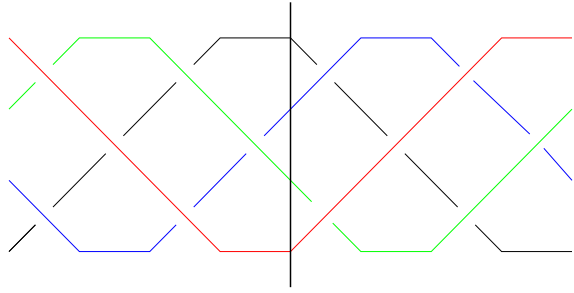
		
	\begin{proof}
		We describe the strands of ${\mathcal B}_N$ in the braid shadow, i.e. their projection to the $xy$.  We endow the plane with a coordinate $Oxy$ such that the $j$-th strand starts at $(0,-j)$ and ends at $(2N,-j)$.  The upper left point has coordinates $(0,0)$ (in Figure \ref{second dessin} it is the starting point of the red strand).\\ We say that a strand is
		{\it ascending}, denoted $\nearrow$ (resp. {\it descending}, denoted $\searrow$) if it has a $+1$ (resp. $-1$) slope. It is {\it horizontal}, denoted $\longrightarrow$, when the slope is $0$. \\
		\\	We describe here the $k$-strands for $k$ odd (the case of an even $k$ is similar): it goes up and down as follows\\
			\begin{enumerate}
				\item
				$\searrow$ from $(0,-k)$ to $(N-1-k, -(N-1))$
				\item
				$\longrightarrow$ from $(N-1-k, -(N-1))$ to $(N-k, -(N-1))$
				\item
				$\nearrow$ from $(N-k,-(N-1))$ to $(2N-1-k, 0)$
				\item
				$\longrightarrow$ from $(2N-1-k, 0)$ to $(2N-k, 0)$
				\item
				$\searrow$  from $(2N-k, 0)$ to $(2N,-k)$ 
			\end{enumerate}

		Assume now that the $k$-th strand is above the $j$-th strand at a crossing point $(x,y)$. Assuming that $j$ is odd (the even case is similar), one of the following two cases occurs
		
		\begin{enumerate}\label{qui est au dessus}
			\item
			$0\leq x\leq N$ and the $k$-th (resp. $j$-th) strand is $\searrow$ (resp. $\nearrow$). Then $S_k$ is as 1. above and $S_j$ is as 3.
			\item
			$N\leq x\leq 2N$ and the $k$-th (resp. $j$-th) strand $\nearrow$ (resp. $\searrow$).  Then $S_k$ is as 3. above and $S_j$ is as 5.
			
		\end{enumerate} 
		In both cases it is easy to check that $k<j$.

	\end{proof}
 
	We conclude the proof of the proposition in the case when $N$ is even; the odd case is similar. We derive from Lemma \ref{au dessus de toutes les suivantes} that for an $n>k$,
	$$A(BA)^n(B^{-1}A^{-1})^nB=A(BA)^{n-k}(BA)^k(B^{-1}A^{-1})^k
	(B^{-1}A^{-1})^{n-k}B$$
	$$=A(BA)^{n-k}(B^{-1}A^{-1})^{n-k}B.$$
	Thus, if $b$ is the remainder of the division of $\frac{q-1}{2}$ by $k$, we have
	\begin{equation}\label{forme finale de la tresse triviale pour N pair}
	B_{N,q,q+N}=A(BA)^b(B^{-1}A^{-1})^bB
	\end{equation}
	The braid (\ref{forme finale de la tresse triviale pour N pair}) is a piece of the braid $(BA)^k(B^{-1}A^{-1})^k$ where the $i$-th strand is above the $j$-th strands, for $j>i$. Thus the same is true for (\ref{forme finale de la tresse triviale pour N pair}) which closes therefore in a trivial knot.

	\subsection{The knot $K(N,1,p)$ is trivial}
		This follows from the Main Theorem. We can also prove it directly by computing the crossing points and their sign: we see that
	 every $\sigma_i^{\pm}$ appears once and only once in the braid $B_{N,1,p}$ and so braid represents a trivial knot.  
	\subsection{The other knots of Proposition \ref{liste de noeuds triviaux}}
	We have now seen two cases where $K(N,q,p)$ is trivial. We know that $K(N,q,p)$ and $K(N,p,q)$ are isotopic; and $K(N,q,k)$ and 
	$K(N,q,2qN+k)$  (resp. $K(N,q,2qN-k)$) are isotopic (resp. mirror image of one another). Thus we can get more examples of trivial knots, e.g. $K(3,5,29)$ and $K(3,5,31)$.
\newpage 
\section{Lists of Jones polynomials}	
 \begin{figure}[!ht]
 	\begin{center} 
		 \includegraphics[scale=.5]{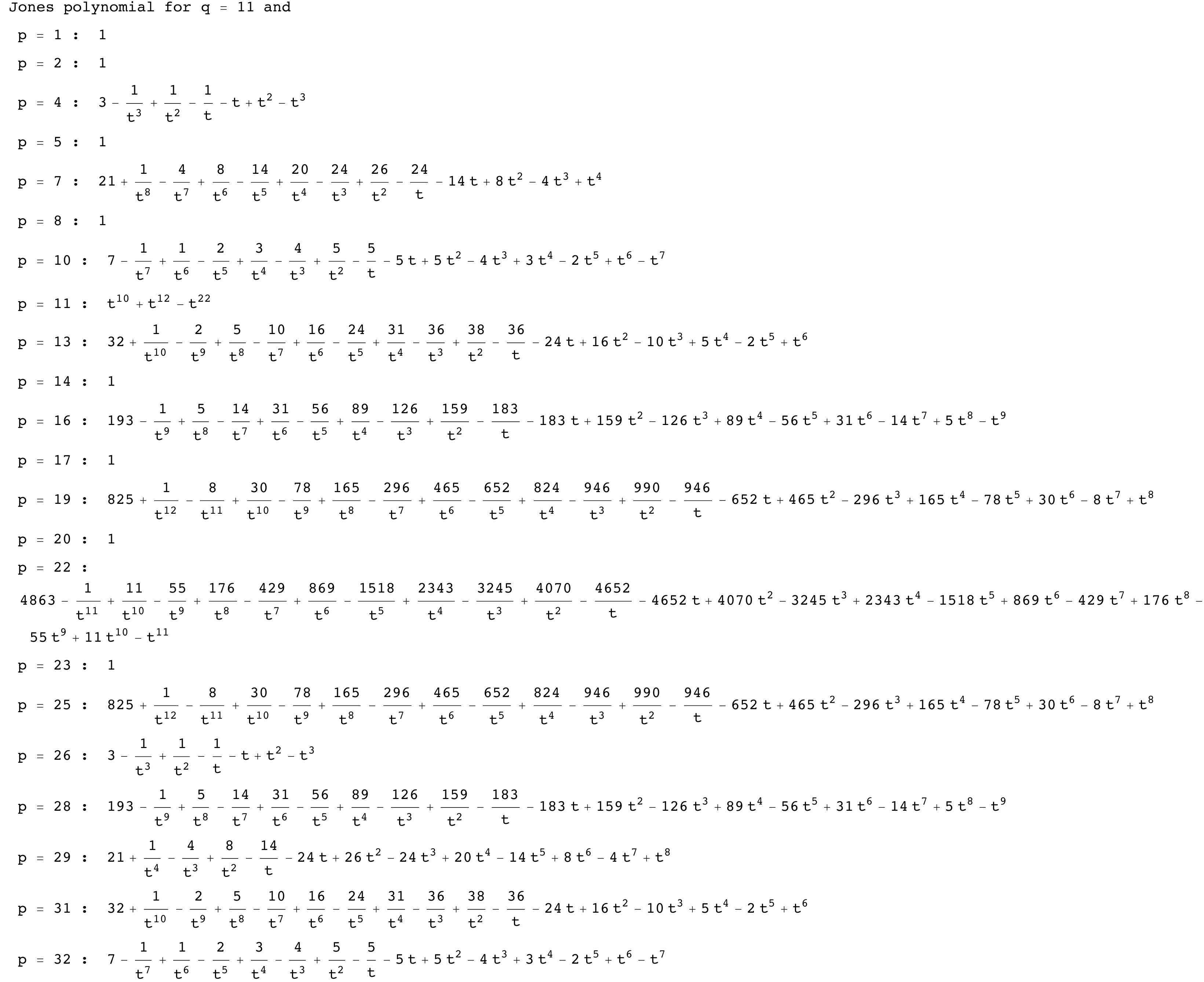} 
 		\caption{List of Jones polynomials of knots $K(3, 11,p)$} 
 	\end{center} \end{figure}
 
  \begin{figure}[!ht]
 	\begin{center} 
		\includegraphics[scale=.38]{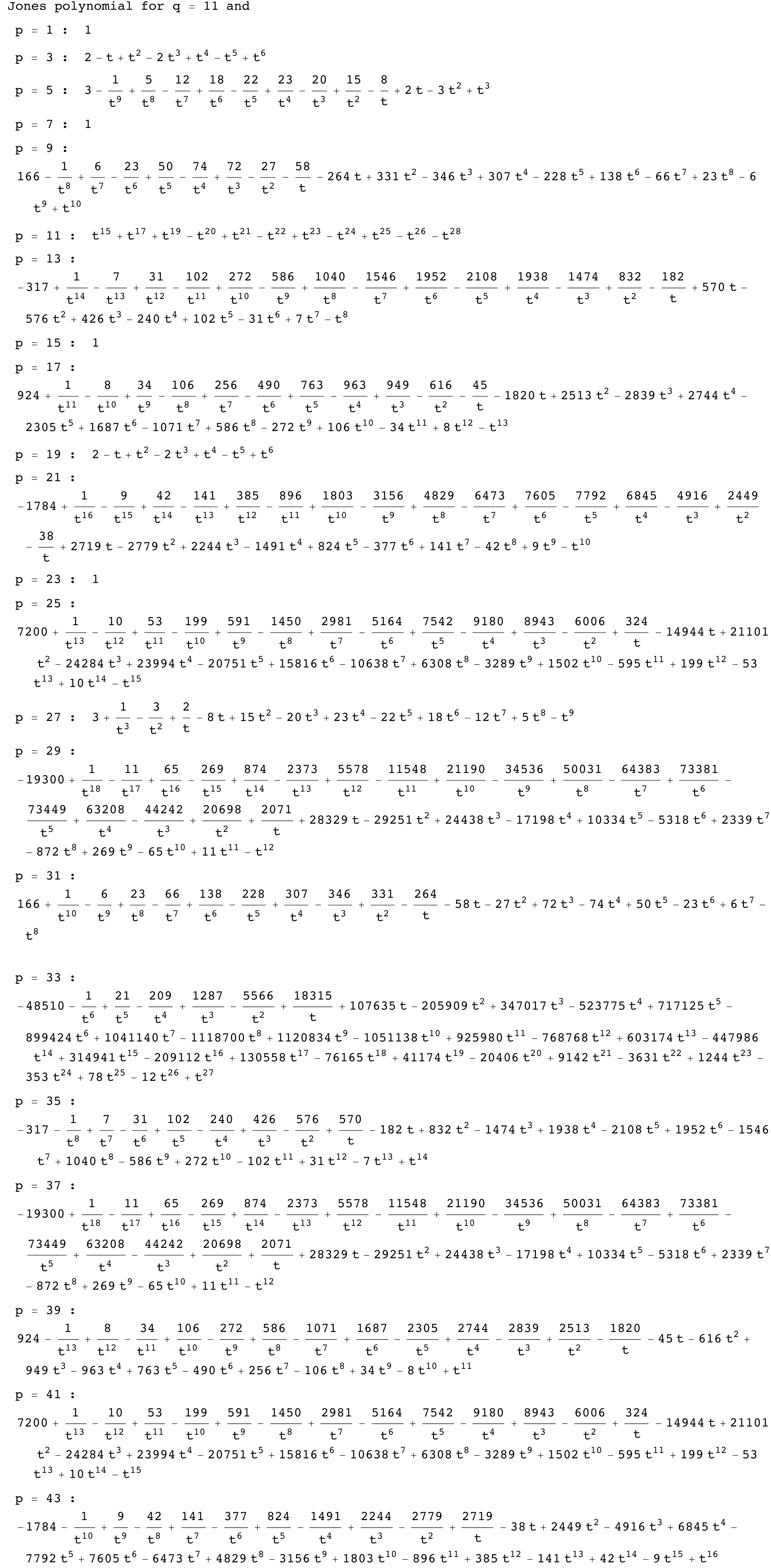} 
 		\caption{List of Jones polynomials of knots $K(4, 11,p)$} 
 	\end{center} \end{figure}
\section{Appendix} 
We give a better proof of the following fact from  [S-V]:
\begin{prop}\label{prop:independance de la phase}Let $\phi_1$ and $\phi_2$ two real numbers. The knots  $K(N,q,p,\phi_1)$ and $K(N,q,p,\phi_2)$ defined in (\ref{definition du disque}) are either isotopic or mirror image of one another.
\end{prop} 
\begin{proof} Without loss of generality, we assume $\phi_1<\phi_2$.\\
	If there is no critical phase (i.e. a phase for which the knot is singular) between $\phi_1$ and $\phi_2$, the two knots are isotopic.\\ In [S-V] we showed that the difference between two critical phases is of the form
	
	\begin{equation}\label{formule:difference entre les phases critiques}
	\frac{N}{2}(\frac{m}{p}+\frac{n}{q})
	\end{equation} for two integers $m,n$. \\
	Thus it is enough to prove that, for a given $\phi_3$, and integers $m$ and $n$, the knots $K(N,q,p,\phi_3)$ and  $K(N,q,p,\phi_3+\frac{N}{2}(\frac{m}{p}+\frac{n}{q}))$ are the same or mirror images of one another.\\ Consider the parametrization of $K(N,q,p,\phi)$ given in (\ref{definition de la tresse}); we change its variable by setting
	\begin{equation} s=t+\frac{Nn}{2q} \end{equation} and we rewrite the expression in (\ref{definition de la tresse}) $$\Big(\sin \frac{2\pi q}{N}(t+k),\cos \frac{2\pi p}{N}\big(t+k+\phi_3+\frac{N}{2}(\frac{m}{p}+\frac{n}{q})\big)\Big)$$
	\begin{equation}
	\label{expression avec m et n}
	=\big((-1)^n\sin \frac{2\pi q}{N}(s+k),(-1)^m\cos \frac{2\pi p}{N}(s+k+\phi_3)\big)
	\end{equation}
	Thus, if $m$ and $n$ have the same (resp. opposite) parities, the two knots are isotopic (resp. mirror image of one another).
\end{proof}

\footnotesize{ Universit\'e F. Rabelais, D\'ep. de Math\'ematiques, 37000 Tours, France,
\\ Marc.Soret@lmpt.univ-tours.fr,
 Marina.Ville@lmpt.univ-tours.fr}


\newpage

\begin{thebibliography}
	{br}\bibitem[br]{br} The Liverpool knot group {\it br9z.p}, https://www.liverpool.ac.uk/~su14/knotprogs.html 
	 \bibitem[B-K]{bk} M. Brandenbursky, J. Kedra {\it Concordance group and stable commutator length in braid groups}, arXiv:1402.3191 (2014), to appear in Algebraic \& Geometric Topology (2015).
\bibitem[B-Z]{bz} G. Burde, H. Zieschang, {\it Knots}, de Gruyter Studies in Math., 2nd Ed. vol. 5, Walter de Gruyter \& Co, New York, 2003. 
\bibitem[Co]{co} J.H. Conway assisted by Fung, F. Y. C. {\it The sensual (quadratic) form}       MAA (1997).
\bibitem[Cr]{cr} P. Cromwell {\it Knots and links} Cambridge University Press, 2014. 
\bibitem[F-W]{f-w} J. Franks, R. Williams {\it Braids and the Jones-Conway polynomial}, Trans. Amer. Math. Soc. 303
    (1987) 97-108


\bibitem[Ka]{k} L.H. Kauffman {\it On knots}, PUP, 1987.
\bibitem[KP]{kp} http://www.knotplot.com/
\bibitem[JP]{jp} V.F.R. Jones,  J.H. Przytycki {\it Lissajous knots and billiard knots}   Knot theory,
Banach Center Publications, 42, Inst. of Maths, Polish Acad. of Sciences, Warsaw  1998                                                                                                                                                                                                                                                                                                                                                                                            
\bibitem[K-M]{km} P. Kronheimer, T. Mrowka, {\it The Genus of Embedded Surfaces in the Projective Plane}, Math. Res. Letters 1, 797-808 (1994) 
\bibitem[K-T]{kt} S. Kinoshita, H. Terasaka {\it On unions of knots},  Osaka J. Math. 9 (1957), 131-153.
\bibitem[L-O]{la0} C.  Lamm, D. Obermeyer {\it Billiard knots in a cylinder} J. Knot Theory and its Ramifications 8(3) (1999) Vol. 353-366.
\bibitem[La 1]{la1}  C.  Lamm 
{\it Deformation of cylinder knots} Fourth chapter of Ph.D. thesis {\it Zylinder-Knoten und symmetrische Vereinigungen} Bonner Mathematische Schriften 321 (1999), http://arxiv.org/pdf/1210.6639 (2012)
\bibitem[La 2]{la} C. Lamm {\it Symmetric unions and ribbons knots}, Osaka  J. Math. 37  (2000),  537-550



 \bibitem[M-O] {mo} X. Mo, R. Osserman, {\it On the Gauss map and total curvature of complete minimal surfaces and an extension of Fujimoto's theorem}, Jour. Diff. geom. 31 (1990), 343-355. \bibitem[M-W]{MW} M. Micallef, B. White, {\it The structure of branch points in minimal surfaces and in pseudoholomorphic curves},  Annals of Maths, 139, (1994), 35-85 
\bibitem[Mi]{mi} J. Milnor, {\it Singular points of complex hypersurfaces}, Ann. of Math. Studies, PUP (1968).

   

  
\bibitem[Ro]{Ro} D. Rolfsen, {\it Knots and Links}, Publish or Perish, Houston, 1990. \bibitem[Ru]{Ru} L. Rudolph, {\it Quasipositivity as an obstruction to sliceness}, Bull. of the AMS,  29 (1) (1993)
    Pages 51-59
\bibitem[S-V]{So-Vi} M. Soret, M. Ville, {\it Singularity knots of minimal surfaces in $\mathbb{R}^4$}, Journal of
    Knot theory and its ramifications, 20 (04) 2011, 513-546.
\end{thebibliography}
\end{document}